\documentclass[a4paper,10pt,reqno]{amsart}
\usepackage{amsmath}
\usepackage{cases}
\usepackage{amsfonts}
\usepackage{latexsym, amssymb, amsmath, amsthm, bbm}
\usepackage[all]{xy}
\usepackage{pgfplots}
\usepackage{enumerate}
\usepackage{mathrsfs}
\usepackage{array}
\usepackage[colorlinks,linkcolor=blue,citecolor=blue]{hyperref}

\DeclareSymbolFont{EulerExtension}{U}{euex}{m}{n}
\DeclareMathSymbol{\euintop}{\mathop} {EulerExtension}{"52}
\DeclareMathSymbol{\euointop}{\mathop} {EulerExtension}{"48}

\allowdisplaybreaks[4]

\def \soc{\operatorname{soc}}
\def \span{\operatorname{span}}
\def \dim{\operatorname{dim}}

\def \Z{\mathbb{Z}}

\def \k{\Bbbk}
\def \A{\mathcal{A}}
\def \B{\mathcal{B}}

\usepackage{mathtools}

\numberwithin{equation}{section}

\newtheorem{theorem}{Theorem}[section]
\newtheorem{lemma}[theorem]{Lemma}
\newtheorem{proposition}[theorem]{Proposition}
\newtheorem{corollary}[theorem]{Corollary}
\newtheorem{definition}[theorem]{Definition}
\newtheorem{example}[theorem]{Example}
\newtheorem{remark}[theorem]{Remark}

\begin{document}
\title[Simple Yetter-Drinfeld modules over infinite-dimensional Taft algebras]{Simple Yetter-Drinfeld modules over infinite-dimensional Taft algebras}
\author[X. Zhen]{Xiangjun Zhen$^\dag$}
\author[G. Liu]{Gongxiang Liu}
\author[J. Yu]{Jing Yu}
\address{Department of Mathematics, Nanjing University, Nanjing 210093, China}
\email{xjzhen@smail.nju.edu.cn}
\address{Department of Mathematics, Nanjing University, Nanjing 210093, China}
\email{gxliu@nju.edu.cn}
\address{School of Mathematical Sciences, University of Science and Technology of China, Hefei 230026, China}
\email{yujing46@ustc.edu.cn}

\thanks{2020 \textit{Mathematics Subject Classification}. 16T05, 16G60 (primary), 16G20, 16G10 (secondary).}
\thanks{Supported by National Key R\&D Program of China 2024YFA1013802 and NSFC 12271243.  }
\thanks{$^\dag$ Corresponding author}
\keywords{Hopf algebras, Infinite-dimensional Taft algebras, Simple Yetter-Drinfeld modules}
\maketitle

\date{}

\begin{abstract}
Let $H$ be an infinite-dimensional Taft algebra over an algebraically closed field $\k$ of characteristic 0. 
We find all the simple Yetter-Drinfeld modules $ V $ over $ H $, and classifies those $ V $ with $ \dim_\k \B(V)< \infty $.
\end{abstract}
\maketitle
\section{Introduction}

Yetter-Drinfeld modules over a bialgebra were introduced by Yetter \cite{Y90} in 1990, where they were called crossed bimodules. For any finite-dimensional Hopf algebra $ H $  over a field $ \k $, Majid \cite{Mj91} identified these modules with the modules over the Drinfeld double $ D(H^{cop}) $ by giving the category equivalences  $ ^H_H \mathcal{YD} \approx _{H^{cop}} \mathcal{YD}^{H^{cop}} \approx _{D(H^{cop})}\mathcal{M} $. The present name comes from \cite{RT}.

Now suppose that $ H $ is a Hopf algebra with bijective antipode. Then $^H_H \mathcal{YD}  $ is a braided monoidal category. This rich structure makes $^H_H \mathcal{YD}$ not only intrinsically interesting, but also plays a crucial role in several areas—such as the construction of new Hopf algebras (e.g., via the Radford biproduct \cite{Ra85}) and the classification of pointed Hopf algebras \cite{AS10}.

A natural and important problem in this setting is the complete classification of the simple objects in $^H_H \mathcal{YD}$. Such a classification is fundamental, for instance, in the classification of pointed Hopf algebras by the lifting method of Andruskiewitsch and Schneider \cite{AS98}.

Many authors have contributed to the classification of simple Yetter–Drinfeld modules. As noted earlier, modules over the Drinfeld double of finite-dimensional Hopf algebras admit straightforward classification. For instance, when $ H $  is a  factorizable Hopf algebra, Reshetikhin and Semenov-Tian-Shansky~\cite{RS88} established that the Drinfeld double $D(H)$ is isomorphic as a Hopf algebra to a twist of $H \otimes H$, facilitating explicit Yetter-Drinfeld module constructions. Schneider~\cite{S01} subsequently proved that $H$ is factorizable \textit{if and only if} $D(H)$ is isomorphic to such a twist. More recently, Liu and Zhu~\cite{LZ19} characterized the structure of simple Yetter-Drinfeld modules over semisimple and cosemisimple quasi-triangular Hopf algebras. When $ \k $ is an algebraically closed field of characteristic zero, Cohen and Westreich~\cite{CW24} provided another characterization.

While the general theory of Yetter–Drinfeld modules is well-developed, the classification of simple Yetter-Drinfeld modules depends heavily on the specific Hopf algebra under consideration. As a result, this problem has  been studied on a case-by-case basis for various specific families of Hopf algebras.


The simplest case occurs when  $ H=\k G $  is the group algebra of a finite group $ G $ over an algebraically closed field $ \k $ of characteristic zero. In this setting, the simple modules over $ D(H) $ are completely described by Dijkgraaf-Pasquier-Roche in \cite{DPR90} and independently by Gould in \cite{G93}. 
Beyond group algebras, the simple Yetter-Drinfeld modules have also been classified for many semisimple Hopf algebras. This includes specific low-dimensional examples, such as the Kac-Paljutkin algebra \cite{HZ07,Shi_2019} and Kashina’s sixteen-dimensional semisimple Hopf algebras \cite{LLK22,ZGH21,ZGH2021,ZY25}, as well as other special cases studied in the literature \cite{S22,S24}.
For non-semisimple Hopf algebras, Yetter-Drinfeld modules have also been classified in important cases. Among pointed Hopf algebras, Taft algebras \cite{C00}  and generalized Taft algebras \cite{MBG21} have been completely calculated. For non-pointed cases, certain specific examples have been studied \cite{GA19,HX2018,X19,X23,X24,XH21,Zheng2024}, though a general classification remains open.

Existing work primarily focuses on finite-dimensional Hopf algebras, where classifications are often achieved by computing representations of the Drinfeld double. For infinite-dimensional Hopf algebras, far fewer results are available, as this method fails. One notable exception is the classification of simple Yetter-Drinfeld modules over the infinite dihedral group~$\mathbb{D}_{\infty}$, established by Zhang \cite{Zhang23}.

The group algebra~$\k \mathbb{D}_{\infty}$ constitutes one of the five classes of affine prime regular Hopf algebras of $\mathrm{GK}$-dimension one, as classified by Ding, Liu, and Wu~\cite{DLW16}. In this article, we study another algebra in this family: the infinite-dimensional Taft algebra. 
Unlike its finite-dimensional counterparts, the infinite-dimensional case exhibits richer representation-theoretic properties and presents distinct challenges due to its non-semisimplicity. Studying the Yetter-Drinfeld modules over this algebra  contributes significantly to the classification of Hopf algebras of GK-dimension one; see  \cite[Question 1.15]{BZ21}.

Our first main result is the following:
\begin{theorem}
Let $ H=H(n,t,\xi) $. Then 
\begin{itemize}
	\item [(1)] All finite-dimensional simple Yetter-Drinfeld modules over $H$ are $ V(ti,j,\lambda) $, where $ i,j \in \Z $ and $ \lambda \in \k $. 
	\item [(2)] The Hopf algebra $ H $ has infinite-dimensional simple Yetter-Drinfeld modules if and only if $ t $ and $ n $ are not coprime. And in this case, all infinite-dimensional simple Yetter-Drinfeld modules over $H$ are $ V(i,j) $, where  $i \in \mathcal{J}$ and $j \in \mathbb{Z}$. 
\end{itemize}
\end{theorem}
See Theorem \ref{thm:YDmodule} for related notions. We also determine the isomorphism relations between these simple modules in Propositions \ref{prop:isocondition} and \ref{prop:isoconditioninfinite}.

Our second main result is the following:
\begin{theorem}
Let $ H=H(n,t,\xi) $. Then 
\begin{itemize}
	\item [(1)] If $ t=0 $, then $ \B(V) $ is infinite-dimensional for every simple Yetter-Drinfeld module $ V $ over $ H $.
	\item [(2)] If $ 1\leq t \leq n-1 $, then the simple Yetter-Drinfeld modules $ V $ over $ H $ for which $ \B(V) $ is finite-dimensional are precisely those listed in Tables 1-6 following Lemma \ref{lem:finiten-nichols-case2}.
\end{itemize}
\end{theorem}
It is worth noting that $ H $ has infinitely many non-isomorphic finite-dimensional simple Yetter-Drinfeld modules, most of which cannot be realized as Yetter-Drinfeld modules over a generalized Taft algebra (see Remark \ref{rmk:cannotrealize}). Moreover, in general $ H $ also possesses infinite-dimensional simple Yetter-Drinfeld modules.

The paper is organized as follows. In Section \ref{section2}, we recall some basics and notations
of Yetter-Drinfeld modules and Nichols algebras.
We also introduce the infinite-dimensional Taft algebras and provide an initial characterization of its modules and comodules. 
In Section \ref{section3}, we prove that every simple Yetter-Drinfeld module over the infinite-dimensional Taft algebras admits a standard basis and describe its comodule structure with respect to this basis.  In Section \ref{section4}, we present explicit constructions of all simple Yetter-Drinfeld modules and clarify their isomorphism classes.  At last, we characterize all modules $ V $ for which the Nichols algebra $ \B(V) $ is finite-dimensional in Section \ref{section5}.


\section{Preliminaries}\label{section2}

Throughout this paper $\k$ denotes an algebraically closed field of characteristic 0 and all spaces are over $\k$. The tensor product over $\k$ is denoted simply by $\otimes$. The set of natural numbers including zero is denoted by $\mathbb{N}_0$, and the positive integers are denoted by $\mathbb{N}$. For any integer $m \geq 2$, the symbol $R_m$ denotes the set of primitive $m$th roots of unity.
The symbol $ H $ will always denote a Hopf algebra over $ \k $ with comultiplication $ \Delta $, counit $ \varepsilon $, and antipode $ S $. We will use the Sweedler’s sigma notation \cite{Swe69} for coproduct and coaction: $ \Delta(h)=\sum h_{(1)}\otimes h_{(2)} $ for coalgebras and $ \delta(v)=\sum v_{(-1)} \otimes v_{(0)} $ for left comodules. The summation sign is often omitted when no explicit computation is involved. Denote by $ G(H)$ the set of group-like elements of $ H $. For $ g,h \in G(H)  $, the linear space of $ (g,h) $-primitives is:
\[
\mathcal{P}_{g,h} (H)= \left\{x\in H \mid \Delta(x)=x\otimes g+h\otimes x \right\}.
\]
We refer to \cite{MonS,Swe69} for the basics about Hopf algebras.

\subsection{Braided vector spaces and Yetter-Drinfeld modules}
We first present the formal definition of a braided vector space.
\begin{definition}
Let $ V $ be a vector space over $ \k $ and $ c\in \mathrm{Aut}(V \otimes V) $. The pair $ (V, c) $ is called a braided vector space if $ c $ satisfies	
\begin{equation*}
	(c\otimes \mathrm{id})(\mathrm{id}\otimes c )(c\otimes \mathrm{id})=(\mathrm{id}\otimes c )(c\otimes \mathrm{id})(\mathrm{id}\otimes c ). 
\end{equation*}
If $ (V, c) $ is a braided vector space, the automorphism $ c $ is called a braiding. If $ (V, c) $ and $ (W, d) $ are braided vector spaces, a  morphism of braided vector spaces $ f : (V, c) \rightarrow (W, d) $ is a linear
map $ f : V \rightarrow W $ with $ (f\otimes f)c = d(f \otimes f) $.
\end{definition}

\begin{example}
{\normalfont 
Let $ I $ be an index set, and let $ (q_{i,j})_{i,j\in I} $ be a family of non-zero scalars in $ \k $. Let $ V $ be a vector space with basis $ \left\{x_i\right\}_{i\in I} $. We define a linear map $ c:V \otimes V \rightarrow V\otimes V $ by
\begin{center}
	$ c(x_i \otimes x_j)=q_{i,j} x_j \otimes x_i \ \ $  for all $ i,j \in I $.
\end{center}
Then $ (V,c) $ is a braided vector space. One says that $ (V, c) $ is a braided vector space of diagonal type. The matrix $ (q_{i,j})_{i,j\in I} $
is called the braiding matrix of $ (V, c) $ with respect to the basis $\left\{x_i\right\}_{i\in I} $.}	
\end{example}

Braidings of diagonal type constitute the simplest class of braidings. Stefan~\cite{ST04} investigated another interesting class called braidings of triangular type, and subsequently established in~\cite{ST07} that such braidings arise from specific Yetter-Drinfeld modules over pointed Hopf algebras with abelian coradical. We now present its formal definition.

\begin{definition}\emph{(}\cite[Definitions 15]{ST04}\emph{)}
Let $ V $ be a finite-dimensional vector space with a totally ordered basis $ X $ and  $ c\in \mathrm{End}(V \otimes V) $. The endomorphism $ c $ will be called right triangular (with respect to the basis $ X $) if for all $ x,y,z \in X $ with $ z>x $ there exist $ \beta_{x,y} \in \k $ and $ w_{x,y,z} \in V$ such that
\begin{equation*}
	c(x\otimes y) = \beta_{x,y} y \otimes x + \sum_{z>x} w_{x,y,z} \otimes z. 
\end{equation*}
A braided vector space $ (V,c) $ will be called right triangular with respect to the basis $ X $ if $ c $ is right triangular with respect to the basis $ X $.
\end{definition}

We now introduce Yetter-Drinfeld modules. Throughout this article, we consider exclusively the left-left variant of these modules.

\begin{definition}
Let $ H $ be a Hopf algebra, 
a Yetter-Drinfeld module over $ H $ is a vector space provided with
\begin{itemize}
	\item [(1)] a structure of left $ H $-module $ \cdot :H\otimes V\rightarrow V $ and
	\item [(2)] a structure of left $ H $-comodule $ \delta : V\rightarrow H \otimes V  $, such that
	\item [(3)] for all $ h\in H $ and $v\in V$,the following compatibility condition holds:	
	\begin{equation}\label{compatibilitycondition}
		\delta(h\cdot v)=h_{(1)}v_{(-1)}S(h_{(3)})\otimes h_{(2)}\cdot v_{(0)}.
	\end{equation}
\end{itemize}
\end{definition}
The category of Yetter-Drinfeld modules over $ H $ is denoted by $ ^H_H \mathcal{YD} $, with morphisms being linear maps that preserve both the action and the coaction.  When the antipode $S$ is bijective, ${}^H_H\mathcal{YD}$ forms a braided monoidal category. Its braiding is given by
\[
c_{V,W} : V \otimes W \to W \otimes V, \quad v \otimes w \mapsto v_{(-1)} \cdot w \otimes v_{(0)}
\]
for all $V, W \in {}^H_H\mathcal{YD}$, $v \in V$, and $w \in W$.

In particular, $ (V, c_{V,V}) $ is a braided vector space. Thus a Yetter-Drinfeld module is always a braided vector space; however, a braided 
vector space could be realized as a Yetter-Drinfeld module over a Hopf algebra with bijective antipode only when the braiding is rigid (see \cite[Definition 4.2.11]{HS20}). This was first proved by Lyubashenko for symmetries \cite{Ly86}; for the general case see \cite{Sch92}.

We conclude this subsection with the following  lemma.

\begin{lemma}\label{lem:compatiable}
Let $ H $ be a Hopf algebra  and $ (V,\cdot,\delta) $  a left $ H $-module and  left $ H $-comodule. If $X \subseteq H$ generates $H$ as an algebra, then the following are equivalent
\begin{itemize}
	\item [(1)] $ (V,\cdot,\delta)$ is a Yetter-Drinfeld module.
	\item [(2)] For all $ h\in X $ and $v\in V$, the compatibility condition \eqref{compatibilitycondition} holds.
\end{itemize}
\end{lemma} 
\begin{proof}
The implication (1) $\Rightarrow$ (2) is immediate. For the converse, define	
\begin{center}
	$ A=\{h\in H \ \vert \   	\delta(h\cdot v)=h_{(1)}v_{(-1)}S(h_{(3)})\otimes h_{(2)}\cdot v_{(0)}      $ for all $ v \in V \} $,
\end{center}
clearly $ A $ is a subspace of $ H $ and $ X\subseteq A $. To show $A = H$, it's sufficient to prove that $ A $ is a subalgebra  since $X$ generates $H$ as an algebra.

It's obvious that $ 1\in A $. For any $ h,k\in A  $ and $v\in V$, we have
\begin{eqnarray*}
	\delta((hk)\cdot v)&=&\delta(h\cdot(k\cdot v)) \\
	&=&h_{(1)}(k\cdot v)_{(-1)}S(h_{(3)})\otimes h_{(2)}\cdot (k\cdot v)_{(0)} \\
	&=&h_{(1)}(k_{(1)}v_{(-1)}S(k_{(3)}))S(h_{(3)})\otimes h_{(2)}\cdot (k_{(2)}\cdot v_{(0)} ) \\
	&=&(h_{(1)}k_{(1)})v_{(-1)}(S(k_{(3)})S(h_{(3)}))\otimes (h_{(2)} k_{(2)})\cdot v_{(0)}  \\
	&=&(h_{(1)}k_{(1)})v_{(-1)}(S(h_{(3)}k_{(3)}))\otimes (h_{(2)} k_{(2)})\cdot v_{(0)}  \\
	&=&(hk)_{(1)}v_{(-1)}S((hk)_{(3)})\otimes (hk)_{(2)}\cdot v_{(0)} 
\end{eqnarray*}
Thus $ hk\in A $ and $ A $ is a subalgebra of $ H $.
\end{proof}

\subsection{Nichols Algebras}
In this subsection, we briefly introduce Nichols algebras. While multiple equivalent definitions exist in the literature, we adopt Heckenberger and Schneider’s formulation from~\cite{HS20}.

\begin{definition}
Let  $C = \bigoplus_{n \in \mathbb{N}_0} C(n)$  be an $\mathbb{N}_0$-graded coalgebra with projections $\pi_n: C \to C(n) $ for all $n \geq 0 $.  For each $n \geq 1$, define the map $\Delta_{1^n}^C$ as the composition:
\begin{equation*}
	\Delta_{1^n}^C : C(n) \subseteq C \xrightarrow{\Delta^{n-1}} C^{\otimes n} \xrightarrow{\pi_1^{\otimes n}} C(1)^{\otimes n},
\end{equation*}	
where $\Delta^{0}=\operatorname{id}_C:C\to C$ and $\Delta^{n}=(\operatorname{id}_C\otimes \Delta^{n-1})\Delta :C \to C^{\otimes (n+1)} $.
\end{definition}

Let $H$ be a Hopf algebra with bijective antipode and $V \in ^H_H \mathcal{YD}$. The tensor algebra $T(V)$ has a natural structure of an  $\mathbb{N}_0$-graded Hopf algebra in $ ^H_H \mathcal{YD}$, where the comultiplication is the algebra map 
$\Delta : T(V) \to T(V) \underline{\otimes} T(V)$ given by $\Delta(v) = v \otimes 1 + 1 \otimes v$, $v \in V$. Note that $T(V)$ is also an $\mathbb{N}_0$-graded coalgebra in the usual sense; thus, $ \Delta_{1^n}^{T(V)}$ can be defined as above. 

\begin{definition}\emph{(}\cite[Definitions 7.1.13]{HS20}\emph{)}
Let $H$ be a Hopf algebra with bijective antipode, and $V \in ^H_H \mathcal{YD}$. Then
\[
\B(V) = T(V) \Big/ \bigoplus_{n \geq 2} \ker\left(\Delta_{1^n}^{T(V)}\right)
\]
is called the Nichols algebra of  $V$. $\mathcal{B}(V) $ is called diagonal type if
$(V, c_{V,V})$ is of digonal type.
\end{definition}
As a vector space, $\B(V)=\bigoplus_{n\geq 0} \mathcal{B}^n(V) $ is $\mathbb{N}_0$-graded, where
\begin{center}
	$ \B^0(V) = \k 1$, $ \B^1(V) = V$, and $\B^n(V) = V^{\otimes n} /  \ker\left(\Delta_{1^n}^{T(V)}\right) $ for all $n\geq 2$.
\end{center}
Moreover, $\B(V)$ is an $\mathbb{N}_0$-graded Hopf algebra in $ ^H_H \mathcal{YD} $
and satisfies the following properties:
\begin{itemize}
	\item $\B^0(V) \cong \k $,  $ \B^1(V)\cong V $  in  $ ^H_H \mathcal{YD} $,
	\item $\B(V)$ is generated as an algebra by $\B^1 (V)$, and
	\item $\B(V)$ is strictly graded, that is, $P(\B(V))=\B^1(V)$.
\end{itemize}
In fact, these important properties characterize  Nichols algebras, see  \cite[Theorem 5.7]{And02}. 

We conclude this subsection with the following remark.

\begin{remark}\label{rmk:nicholalgebra}\hfil
{\normalfont 
\begin{itemize}
	\item[(1)] For all $n\geq 2$, let $S_n:V^{\otimes n}\to V^{\otimes n}$ be the braided symmetrizer maps associated to the braided vector space $(V,c_{V,V})$, see \cite[Definitions 1.8.10]{HS20}. By  \cite[Corollary 1.9.7]{HS20},  $\Delta_{1^n}^{T(V)}=S_n$ in $\mathrm{End}(V^{\otimes n})$. Moreover, the Nichols algebra $\B(V)$ as an algebra and a coalgebra only depends on the braided vector space $(V,c_{V,V})$.
	\item[(2)] If $V$ contains a non-zero element $v$ such that $c(v \otimes v) = v \otimes v$, then $\dim_{\k}(\B(V))=\infty$; see \cite[Example 7.1.3]{HS20}.
\end{itemize}}
\end{remark}

\subsection{Infinite-dimensional Taft algebras}
In this subsection, we briefly introduce the infinite-dimensional Taft algebras.

\begin{definition}\label{def:infinite-Taft}
Let $ n $ and $ t $ be positive integers with $ n\geq 2 $ and $ 0\leq t\leq n-1 $, and let $ \xi $ be a primitive $ n $-th root of 1. Let $ H:=H(n,t,\xi) $ be the $ \k $-algebra generated by x and g subject to the relations
\begin{center}
	$ g^n=1 $ \  and  \ $ xg=\xi gx $.
\end{center}
The coalgebra structure of H is defined by
\begin{center}
	$ \Delta(g)=g\otimes g \  , \  \varepsilon(g)=1 $
\end{center}
and
\begin{center}
	$ \Delta(x)= x\otimes g^t + 1\otimes x \  , \  \varepsilon(x)=0 $.
\end{center}
The antipode S of H is defined by 
\begin{center}
	$ S(g)=g^{-1} $ \ and  \ $ S(x)=-xg^{-t} $.
\end{center}

\end{definition}

\begin{remark}\label{rmk:finiteTaft}\hfill
{\normalfont 
\begin{itemize}
	\item [(1)] The algebra above is a pointed Hopf algebra whose group of group-like elements is 
	$G(H) =\left\langle g \right\rangle=   \{1, g, g^2, \dots, g^{n-1}\}$. 
	A linear basis is given by 
	\[
	\left\{ g^i x^j \mid 0 \leq i < n,\  j \geq 0 \right\}.
	\]
	\item [(2)] When $ t $ and $ n $ are coprime, the quotient of the algebra $ H(n,t,\xi) $ by the relation $ x^n=0 $ again forms a Hopf algebra, denoted $ H(n,t,\xi,0)  $, under the original comultiplication and antipode. This quotient algebra  $ H(n,t,\xi,0) $ is precisely the $ n^2 $-dimensional Taft algebra \cite{T71}.
\end{itemize}}
\end{remark}

\subsection{Modules and comodules over infinite-dimensional Taft algebras}
For the remainder of this paper, fix parameters $n \geq 2$, $0 \leq t \leq n-1$, and let $\xi$ be a primitive $n$-th root of unity. Set $H = H(n, t, \xi)$ and denote its group of group-like elements by $G = G(H)$. 
Set $w = \xi^t$ and  $ N=o(\xi^t)=o(g^t) $. Then for all $i, j \in \Z$, $ i \equiv j \pmod{N} $ if and only if $ ti \equiv tj \pmod{n} $.

By convention, for any integer $i \in \Z$, let $[i]$ denote the congruence class of $i$ modulo $n$. 
We begin by introducing the following concepts.

\begin{definition}
Let $(V, \cdot, \delta)$ be a left $H$-module and left $H$-comodule, $W \subseteq V$ a subspace, and $i \in \Z$. Define
\[
W_{[i]} = \{ v \in W \mid g \cdot v = \xi^i v \}
\quad \text{and} \quad 
W^{[i]} = \{ v \in W \mid \delta(v) = g^i \otimes v \}.
\]
\end{definition}

These subspaces are well-defined: when $i \equiv j \pmod{n}$, 
$W_{[i]} = W_{[j]}$ and $W^{[i]} = W^{[j]}$. 
The following properties are immediate.

\begin{lemma}\label{lem:subspace-properties}
Let $(V, \cdot, \delta)$ be a left $H$-module and left $H$-comodule, $W \subseteq V$ a subspace, and $i, j \in \Z$. Then:
\begin{itemize}
	\item [(1)]  $W_{[i]}$ is a $\k G$-submodule of $V$, and $W^{[i]}$ is an $H$-subcomodule of $V$.
	\item [(2)]  $(W_{[i]})^{[j]} = (W^{[j]})_{[i]} = \left\{ v \in W \mid g \cdot v = \xi^i v \text{ and } \delta(v) = g^j \otimes v \right\}$.
\end{itemize}
\end{lemma}

Denoting the common space by $W_{[i]}^{[j]} \coloneqq (W_{[i]})^{[j]} = (W^{[j]})_{[i]}$, we note that if $i_1 \equiv i_2 \pmod{n}$ and $j_1 \equiv j_2 \pmod{n}$, then $W_{[i_1]}^{[j_1]} = W_{[i_2]}^{[j_2]}$. For a simple Yetter-Drinfeld module $V$ over $H$, every non-zero element in $V_{[i]}^{[j]}$ generates $V$ as an $H$-module (see Lemmas~\ref{thm:standbasis} and~\ref{thm:standbasisinfinite}). This fact motivates the following definition.

\begin{definition}\label{def:standardelement}
Let $(V, \cdot, \delta)$ be a left $H$-module and left $H$-comodule. A nonzero element $v \in V_{[i]}^{[j]}$ for some $i, j \in \mathbb{Z}$ is called a {\normalfont standard element} of $V$. In this case, we say $v$ is a standard element of type $(j,i)$.
\end{definition}

The following lemma is a direct consequence of Maschke's Theorem.

\begin{lemma}\label{lem:kGmoduledecomposition}
Let  $ V $ be a  left $ \k G $-module, then $ V=\bigoplus\limits_{i=1}^n V_{[i]}  $. In particular, if $ V $ is a left $ H $-module, then $ V=\bigoplus\limits_{i=1}^n V_{[i]}  $.
\end{lemma}

We now characterize the structure of  $H$-modules.

\begin{proposition}\label{prop:structureofHmodule}
A left $H$-module structure consists of a pair $(V, T)$, where $V$ is a  left $\k G$-module  and  $T : V \to V$ is a linear map such that $T(V_{[i]}) \subseteq V_{[i-1]}$ for all $i \in \Z$.
Conversely, any such pair defines a left $H$-module structure.
\end{proposition}

\begin{proof}
Suppose $V$ is a left $H$-module, which is automatically a $\k G$-module. Define $T: V \to V$ by $T(v) = x \cdot v$. For any $i \in \mathbb{Z}$ and $v \in V_{[i]}$:	
\begin{eqnarray*}
	g \cdot (T(v))&=&(gx)\cdot v\\
	&=&(\xi^{-1}xg)\cdot v \\
	&=&\xi^{i-1} (x\cdot v) \\
	&=&\xi^{i-1}T(v),
\end{eqnarray*}
confirming $T(V_{[i]}) \subseteq V_{[i-1]}$.

Conversely, given a pair $(V, T)$ with $V$ a  left $\k G$-module and $T: V \to V$ satisfying $T(V_{[i]}) \subseteq V_{[i-1]}$ for all $i \in \mathbb{Z}$, define the $H$-action by
\[
(g^i x^j) \cdot v \coloneqq g^i \cdot T^j(v)
\]
for $0 \leq i \leq n-1$, $j \geq 0$, and $v \in V$. 

Fix $v \in V_{[m]}$ for some $m \in \Z$. The identity $1 \cdot v = v$ holds trivially. For any $0 \leq i,k \leq n-1$ and $j,l \geq 0$, let $p \equiv i+k \pmod{n}$ with $0 \leq p \leq n-1$. Then
\begin{eqnarray*}
	((g^ix^j)(g^kx^l)) \cdot v &=&(\xi^{jk}g^{i+k}x^{j+l}) \cdot v \\
	&=&(\xi^{jk}g^{p}x^{j+l}) \cdot v \\
	&=&\xi^{jk+p(m-(j+l))} T^{j+l}(v)
\end{eqnarray*}
and
\begin{eqnarray*}
	(g^ix^j)\cdot ((g^kx^l) \cdot v)
	&=&\xi^{k(m-l)}(g^ix^j)\cdot  T^l(v)\\
	&=&\xi^{k(m-l)}\xi^{i(m-(j+l))} T^{j+l}(v) \\
    &=&\xi^{jk+(i+k)(m-(j+l))} T^{j+l}(v) \\
    &=&\xi^{jk+p(m-(j+l))} T^{j+l}(v).
\end{eqnarray*}
The equality $((g^i x^j)(g^k x^l)) \cdot v = (g^i x^j) \cdot ((g^k x^l) \cdot v)$ follows, confirming $V$ is a left $H$-module.
\end{proof}

The following lemma will be used to establish the simplicity of modules constructed in Section~\ref{section4}. Its proof follows from standard linear algebra techniques and is omitted.

\begin{lemma}\label{lem:kGmoduleproperty}
Let  $ V $ be a  left $ \k G $-module 
 and $ W $ be a submodule of $ V $. If $ \sum\limits_{i=1}^n v_i \in W$, where $ v_i \in V_{[i]} $  for each $ i $, then $ v_i\in W $ for each $ i $.
\end{lemma}

If $ V $ is a comodule, define the socle of $ V $, $ \soc (V) $ to be the sum of all simple subcomodules of $ V $. Since every non-zero comodule contains a simple subcomodule, $\soc(V) \neq 0$ whenever $V \neq 0$. For  any Yetter-Drinfeld module $ V $ over $H$, $\soc(V)$ is in fact spanned by its standard elements (Definition~\ref{def:standardelement}).

\begin{lemma}\label{lem:socV}
Let  $ V $ be a Yetter-Drinfeld module over $ H $, then 
\begin{itemize}
	\item [(1)]  $ V^{[i]}  $ is a $ \k G $-submodule for each $ i\in \Z $;
	\item [(2)] $ \soc (V)=\bigoplus\limits_{1\leq i,j\leq n} V_{[j]}^{[i]} $.
\end{itemize} 
In particular, $V \neq 0$ implies $V_{[j]}^{[i]} \neq 0$ for some $i,j \in \Z$, so $V$ contains standard elements.
\end{lemma}

\begin{proof}
For $i \in \mathbb{Z}$ and $v \in V^{[i]}$, the coaction satisfies:
\[
\delta(g \cdot v) = g g^i S(g) \otimes g \cdot v = g^i \otimes g \cdot v,
\]
confirming $g \cdot v \in V^{[i]}$. Thus $V^{[i]}$ is a $\k G$-submodule, establishing (1).

Since $H$ is pointed, every simple $H$-comodule is one-dimensional and corresponds to a group-like element. By Lemma~\ref{lem:kGmoduledecomposition} and part (1):
$$\soc (V)=\bigoplus\limits_{i=1}^n V^{[i]}=\bigoplus\limits_{i=1}^n \left(\bigoplus\limits_{j=1}^n V_{[j]}^{[i]}\right)=\bigoplus\limits_{1\leq i,j\leq n} V_{[j]}^{[i]},$$
proving (2). 

Finally, $V \neq 0$ implies $\soc(V) \neq 0$, so $V_{[j]}^{[i]} \neq 0$ for some $i,j \in \mathbb{Z}$.
\end{proof} 

We conclude this subsection with the notion of a \emph{comatrix}, a coalgebraic analog of the matrix representation for linear maps.

\begin{definition}
Let $(C,\Delta,\varepsilon)$ be a coalgebra over $\k$, a square matrix $\A=(c_{ij})_{r\times r}$ over $C$ is said to be a comatrix, if for any $1\leq i,j \leq r$, we have
\begin{center}
	$\Delta(c_{ij})=\sum\limits_{l=1}^r c_{il}\otimes c_{lj}$ \ and \ $\varepsilon(c_{ij})=\delta_{i, j}$, 
\end{center}
where $\delta_{i, j}$ denotes the Kronecker notation.
\end{definition}

\begin{lemma}\emph{(}\cite[Lemma 2.1.1]{HS20}\emph{)}\label{lem:comoduleequicodition}
Let $(C,\Delta,\varepsilon)$ be a coalgebra over $\k$ and $ V $ an $n$-dimensional vector space with basis $\{v_1, \dots, v_n\}$. For linear map $\delta : V \to C \otimes V$ defined by
\[
\delta(v_k) = \sum_{l=1}^n c_{kl} \otimes v_l \quad (1 \leq k \leq n, \  c_{kl} \in C),
\]
let $ \A=(c_{ij})_{n\times n} $. Then $ (V,\delta) $ is a left $ C $-comodule if and only if $ \A $ is a comatrix.
\end{lemma}

\section{Comodule Structures over Simple Yetter-Drinfeld Modules}\label{section3}

In this section, we establish that every simple Yetter-Drinfeld module over $H$ admits a standard basis and compute its comatrix relative to this basis. Analysis of these comatrices reveals fundamental properties of simple Yetter-Drinfeld modules over $H$, culminating in subsection \ref{subsection3}.

\subsection{Existence of Standard Bases for Simple Yetter-Drinfeld Modules }\label{subsection31}

Let $V$ be a Yetter-Drinfeld module over $H$. For any $v \in V_{[j]}^{[i]}$ with $i, j \in \mathbb{Z}$, define $V(v,k)$ to be the linear span of $\{v, x \cdot v, \dots, x^k \cdot v\}$ for $k \geq 0$, and $V(v,\infty)$ to be the linear span of $\{x^k \cdot v \mid k \geq 0\}$.

\begin{proposition}\label{pro:spaniscomodule}
Let $V$ be a  Yetter-Drinfeld module over $H$ and take $v \in V_{[j]}^{[i]}$ for $i,j \in \Z$. Then for any $k \geq 0$:
\begin{itemize}
	\item [(1)] $x^k \cdot v \in V_{[j-k]}$;
	\item [(2)] $V(v,k)$ is a subcomodule of $V$;
	\item [(3)] If $x^{k+1} \cdot v \in V(v,k)$, then $V(v,k)$ is a Yetter-Drinfeld submodule of $V$;
	\item [(4)] $V(v,\infty)$ is a Yetter-Drinfeld submodule of $V$.
\end{itemize}
\end{proposition}

\begin{proof}
(1) follows directly from Proposition \ref{prop:structureofHmodule}.

For (2), we proceed by induction on $k$. The case $k=0$ holds trivially. Assume $V(v,k)$ is a subcomodule. Then $\delta(x^k \cdot v) = \sum_{l=0}^k c_{kl} \otimes (x^l \cdot v)$ for some $c_{kl} \in H$. By the compatibility condition \eqref{compatibilitycondition}, 	
\begin{eqnarray*}
	& &\delta (x^{k+1} \cdot v) \\
	&=&\delta (x\cdot (x^{k} \cdot v)) \\
	&=&\sum\limits_{l=0}^k \left(c_{kl}S(x)\otimes 1\cdot (x^l \cdot v)+ c_{kl}S(g^t)\otimes x\cdot (x^l \cdot v)+xc_{kl}S(g^t)\otimes g^t\cdot (x^l \cdot v) \right) \\
	&=&\sum\limits_{l=0}^{k+1} c_{k+1,l} \otimes x^l \cdot v,
\end{eqnarray*}
where
\begin{equation}\label{coefficientckl}
	c_{k+1,l}=\begin{cases}
		c_{k0}S(x)+\xi^{tj}xc_{k0}S(g^t),     & l=0, \\[5pt]
		c_{kl}S(x)+\xi^{t(j-l)}xc_{kl}S(g^t) +    c_{k,l-1}S(g^t),     \qquad    &     0<l<k+1,  \\[5pt]
		c_{kk} S(g^t),           &     l=k+1.
	\end{cases}
\end{equation}
Thus $\delta(x^{k+1} \cdot v) \in H \otimes V(v,k+1)$. By the induction hypothesis, $V(v,k+1)$ is a subcomodule.

For (3), note that by (1), $V(v,k)$ is a $\k G$-submodule. If $x^{k+1} \cdot v \in V(v,k)$, then $V(v,k)$ is $x$-stable, hence an $H$-submodule. Combined with (2), it is a Yetter-Drinfeld submodule. 

By a similar argument, one can prove (4).
\end{proof}

Note that every non-zero  Yetter-Drinfeld module over $H$ contains standard elements by Lemma \ref{lem:socV}.

\begin{lemma}\label{lem:standbasis}
Let $V$ be a finite-dimensional simple Yetter-Drinfeld module over $H$, $v \in V$ a standard element, and $k \geq 0$ an integer. The following are equivalent:
\begin{itemize}
	\item[(1)] $\{v, x \cdot v, \dots, x^k \cdot v\}$ is a basis of $V$.
	\item[(2)] $k$ is the largest integer such that $\{v, x \cdot v, \dots, x^k \cdot v\}$ is linearly independent.
\end{itemize}
\end{lemma}

\begin{proof}
Since (1) $\Rightarrow$ (2) is immediate, we prove (2) $\Rightarrow$ (1). Assume $k$ is the largest integer with $\{v, x \cdot v, \dots, x^k \cdot v\}$ linearly independent. Then $x^{k+1} \cdot v \in V(v,k)$. By Proposition \ref{pro:spaniscomodule}(3), $V(v,k)$ is a Yetter-Drinfeld submodule. As $V$ is simple, $V = V(v,k)$. Thus $\{v, x \cdot v, \dots, x^k \cdot v\}$ is a basis of $V$.
\end{proof}

\begin{lemma}\label{thm:standbasis}
Let $V$ be a simple Yetter-Drinfeld module over $H$ with $\dim_{\k} V = p+1$ for $p \geq 0$. For any standard element $v \in V$, the set $\{v, x\cdot v, \dots, x^p \cdot v\}$ forms a basis of $V$. 
\end{lemma}

\begin{proof}
Let $k$ be the largest integer such that $\{v, x\cdot v, \dots, x^k \cdot v\}$ is linearly independent. By Lemma \ref{lem:standbasis}, this set constitutes a basis of $V$. Hence $k+1 = \dim_{\k} V = p+1$, yielding $k=p$ and completing the proof. 
\end{proof}

\begin{lemma}\label{thm:standbasisinfinite}
Let $V$ be an infinite-dimensional simple Yetter-Drinfeld module over $H$. For any standard element $v \in V$, the set $\{x^k \cdot v \mid  k\geq 0   \}$ forms a basis of $V$. 	
\end{lemma}

\begin{proof}
Since $V$ is simple, Proposition \ref{pro:spaniscomodule}(4) implies $V = V(v,\infty)$. To complete the proof, it suffices to show that the set $\{x^k \cdot v \mid k \geq 0\}$ is linearly independent. 

Suppose, for contradiction, that this set is linearly dependent. Then there exists some $k \geq 0$ such that $x^{k+1} \cdot v \in V(v,k)$. By Proposition \ref{pro:spaniscomodule}(3), $V(v,k)$ forms a Yetter-Drinfeld submodule of $V$. Since $V$ is simple, this implies $V = V(v,k)$. However, this contradicts the assumption that $V$ is infinite-dimensional, as $V(v,k)$ is finite-dimensional by construction.	
\end{proof}

\begin{remark}
{\normalfont 
The bases in Lemmas \ref{thm:standbasis} and \ref{thm:standbasisinfinite} are referred to as} standard bases. 
\end{remark}

The action and coaction of a simple Yetter-Drinfeld module admit a simple form with respect to a standard basis. We take the case of finite-dimensional simple Yetter-Drinfeld modules as an example to illustrate this.

Let  $ V $  be a simple Yetter-Drinfeld module over $H$ of dimension  $ p+1 $  for some  $ p \geq 0 $. Fix a standard element  $ v \in V_{[j]}^{[i]} $  with $  i, j \in \mathbb{Z} $. Then the set $\{v, x\cdot v, \dots, x^p \cdot v\}$ forms a standard basis of $  V  $. Denoting  $ v_k = x^k \cdot v $  for  $ 0 \leq k \leq p $, the action and coaction on this basis are given by:

\begin{itemize}
	\item The action of  $ g $  is given for $  0 \leq k \leq p $ by
	\begin{equation*}
		g\cdot v_k=\xi^{j-k} v_k.
	\end{equation*}
	\item The action of  $ x $  is given for $  0 \leq k \leq p $ by
	\begin{equation}\label{standbasisactioncoaction}
		x\cdot v_k=
		\begin{cases}
			v_{k+1} , & 0\leq k<p  ,\\[5pt]
			\sum_{l=0}^{p} a_l v_l    ,   & k=p  ,
		\end{cases}
	\end{equation}
    where $ a_l \in \k $ for $ 0\leq l \leq p $.
	\item The coaction  is given for $  0 \leq k \leq p $ by
	\begin{equation*}
		\delta(v_k) =\sum\limits_{l=0}^{k} c_{k,l} \otimes v_l,
	\end{equation*}
    where the coefficients  $ c_{k,l} $  are defined recursively by Equation~\eqref{coefficientckl}, with initial condition  $ c_{0,0} = g^i $.
\end{itemize}

\subsection{The Comatrix Relative to a Standard Basis}

The preceding subsection concluded with a description of the action and coaction on a simple Yetter-Drinfeld module relative to a standard basis. We now examine the converse problem: given arbitrary parameters $ i,j\in \Z $, $ p\geq 0 $, and scalars $ a_l \in \k $ ($ 0\leq l \leq p $), let $ V $ be a vector space of dimension $ p+1 $ with basis $\{v_0, v_1, \dots, v_p\}$, endowed with the action and coaction defined in \eqref{standbasisactioncoaction}.  A natural question is whether $ V $ thereby becomes a Yetter-Drinfeld module over $ H $.

In general, the answer is negative. A necessary condition for $ V $ to be an $H$-module is that $ x\cdot v_k \in V_{[j-(k+1)]} $, which forces $ a_l=0 $ whenever $ j-l \notin [j-(k+1)] $. Moreover, the parameters $ p $ and $ a_l $  must be chosen to satisfy the compatibility conditions \eqref{compatibilitycondition}.

Although $ V $ may not be a full Yetter-Drinfeld module, we will prove that it always admits an $ H $-comodule structure. Our strategy is to compute the coefficient matrix of the coaction explicitly and verify that it is a comatrix. The entries of this matrix are defined recursively by Equation \eqref{coefficientckl}, as we now formalize.

\begin{definition}\label{def:ckl}
For $i,j \in \mathbb{Z}$, define coefficients $c_j^i(k,l) \in H$ recursively as follows: set $c_j^i(0,0) := g^i$ and for $k \geq 0$, $0 \leq l \leq k+1$:
\begin{equation*}
	c_j^i(k+1,l):=\begin{cases}
		c_j^i(k,0)S(x)+\xi^{tj}xc_j^i(k,0)S(g^t),     & l=0 ,\\[5pt]
		c_j^i(k,l)S(x)+\xi^{t(j-l)}xc_j^i(k,l)S(g^t) +    c_j^i(k,l-1)S(g^t),     \qquad    &     0<l<k+1,  \\[5pt]
		c_j^i(k,k) S(g^t)  ,         &     l=k+1.
	\end{cases}
\end{equation*} 
For $p \geq 0$, define the $(p+1) \times (p+1)$ lower triangular matrix $\mathcal{A}_j^i(p)$ with entries:
\[
\big( \mathcal{A}_j^i(p) \big)_{k,l} = 
\begin{cases} 
	c_j^i(k,l) , & 0 \leq l \leq k \leq p , \\
	0 , & \text{otherwise}.
\end{cases}
\]
Explicitly:
\[
\mathcal{A}_j^i(p) = 
\begin{pmatrix}
	c_j^i(0,0) & 0 & \cdots & 0 \\
	c_j^i(1,0) & c_j^i(1,1) & \cdots & 0 \\
	\vdots & \vdots & \ddots & \vdots \\
	c_j^i(p,0) & c_j^i(p,1) & \cdots & c_j^i(p,p)
\end{pmatrix}.
\]
\end{definition}
The following properties are immediate from the definitions.
\begin{remark}\hfill
{\normalfont 
\begin{itemize}
\item[(1)]  If  $ i_1 \equiv i_2 \pmod n  $ and  $ j_1 \equiv j_2 \pmod N  $, then $ \A_{j_1}^{i_1}(p) =\A_{j_2}^{i_2}(p)  $ for all $ p\geq 0 $.
\item[(2)] For any Yetter-Drinfeld module $V$ over $H$ and $v \in V_{[j]}^{[i]}$,
\begin{equation}\label{deltax^kv}
	\delta(x^k\cdot v)=\sum\limits_{l=0}^k c_j^i(k,l) \otimes x^l \cdot v 
\end{equation}
holds for all $k \geq 0$.
\end{itemize}}
\end{remark}

To express the entries $c_j^i(k,l)$ explicitly, we define the following families over $\k$.

\begin{definition}
For $i,j \in \mathbb{Z}$, define the following coefficient families in $\k$:
	\begin{itemize}
		\item[(1)] For $k \geq 0$ and $0 \leq l \leq k$:
		 \begin{equation*}
		  R_j^i(k,l):=\begin{cases}
			\xi^{t(j-l)}-\xi^{t(k-1)-i},  \qquad  & 0\leq l <k, \\[5pt]
			1 ,          &     l=k .
		             \end{cases}
	     \end{equation*} 
		
		\item[(2)] Set $\lambda_j^i(0,0) := 1$, and for $k \geq 1$, $0 \leq l \leq k$:
        \begin{equation*}
        	\lambda_j^i(k,l):=\begin{cases}
        		R_j^i(k,0)\lambda_j^i(k-1,0) , & l=0, \\[5pt]
        		R_j^i(k,l)\lambda_j^i(k-1,l)+\lambda_j^i(k-1,l-1) , \qquad  & 0< l <k , \\[5pt]
        		1  ,         &     l=k .
        	                  \end{cases}
        \end{equation*} 
	\end{itemize}
\end{definition}

The following equality holds immediately.

\begin{proposition}\label{pro:cijformula}
For  $i, j \in \mathbb{Z}$ and $0 \leq l \leq k$, $ c_j^i(k,l)=\lambda_j^i(k,l)x^{k-l}g^{i-kt} $.
\end{proposition}

\begin{proof}
We proceed by induction on $ k $. The base case $k=0$ holds by definition. Assume inductively that  $ c_j^i(k,l)=\lambda_j^i(k,l)x^{k-l}g^{i-kt} $ for all $0 \leq l \leq k$. 

For $l = k+1$, Definition \ref{def:ckl} immediately yields $ c_j^i(k+1,k+1)=g^{i-(k+1)t} $. For $0 < l < k+1$, we have
\begin{eqnarray*}
	& & c_j^i(k+1,l) \\
	&=&c_j^i(k,l)S(x)+\xi^{t(j-l)}xc_j^i(k,l)S(g^t) +    c_j^i(k,l-1)S(g^t)  \\[5pt]
	&=&\left(\lambda_j^i(k,l)x^{k-l}g^{i-kt}\right) S(x)+\xi^{t(j-l)}x\left(\lambda_j^i(k,l)x^{k-l}g^{i-kt}\right)S(g^t)+\left(\lambda_j^i(k,l-1)x^{k-l+1}g^{i-kt}\right)S(g^t)  \\[5pt]
	&=&\left(-\xi^{tk-i} \lambda_j^i(k,l)+\xi^{t(j-l)}\lambda_j^i(k,l) + \lambda_j^i(k,l-1) \right)x^{k+1-l}g^{i-(k+1)t} \\[5pt]
	&=&\left(R_j^i(k+1,l) \lambda_j^i(k,l) + \lambda_j^i(k,l-1) \right)x^{k+1-l}g^{i-(k+1)t} \\[5pt]
	&=&\lambda_j^i(k+1,l)x^{k+1-l}g^{i-(k+1)t}. 
\end{eqnarray*}

Similarly, $ c_j^i(k+1,0)=\lambda_j^i(k+1,0)x^{k+1}g^{i-(k+1)t} $   follows by analogous computation. This completes the inductive step and the proof.	
\end{proof}

We now verify that $\mathcal{A}_j^i(m)$ admits a comatrix structure, using the following characterization.

\begin{lemma}\label{lem:comatrixcriteration}
For  $ i,j\in \Z $ and $ m\geq 0 $, the following are equivalent:
\begin{itemize}
	\item [(1)] $ \A_j^i(m) $ is a comatrix.
	
	\item [(2)] For all $ 0\leq l\leq k \leq m $ and $ 0\leq p \leq k-l $, 
	\begin{equation*}
		\lambda_j^i(k,p+l)\lambda_j^i(p+l,l)=\binom{k-l}{p}_{\xi^t} \xi^{-t(k-(p+l))p}\lambda_j^i(k,l) .
	\end{equation*}
\end{itemize}
	
\end{lemma}

\begin{proof}
Set $c_j^i(k,l) = 0$ for $l > k$. We prove (1) $\Rightarrow$ (2), as the converse follows similarly.	
Fix $0 \leq l \leq k \leq m$. By the comatrix property:	
\begin{eqnarray*}
	\Delta(c_j^i(k,l)) &=&\sum\limits_{p=0}^m c_j^i(k,p) \otimes c_j^i(p,l)\\
	&=&\sum\limits_{p=l}^k c_j^i(k,p) \otimes c_j^i(p,l) \\
	&=&\sum\limits_{p=l}^k \lambda_j^i(k,p) \lambda_j^i(p,l) x^{k-p}g^{i-kt} \otimes x^{p-l}g^{i-pt}. 
\end{eqnarray*}
Alternatively, expanding $\Delta(c(k,l))$ directly:
\begin{eqnarray*}
	& &\Delta(c_j^i(k,l))  \\[5pt]
	&=&\lambda_j^i(k,l) \Delta(x)^{k-l} \Delta(g)^{i-kt} \\[5pt]
	&=&\lambda_j^i(k,l) (x\otimes g^t +1 \otimes x)^{k-l}  (g\otimes g)^{i-kt}\\[5pt]
	&=&\lambda_j^i(k,l) \left(\sum\limits_{p=0}^{k-l} \binom{k-l}{p}_{\xi^t}(x\otimes g^t)^{k-l-p} (1\otimes x)^p \right)  (g^{i-kt}\otimes g^{i-kt})\\[5pt]
	&=&\sum\limits_{p=0}^{k-l} \lambda_j^i(k,l) \binom{k-l}{p}_{\xi^t} \xi^{-t(k-(p+l))p} \left(x^{k-(p+l)}g^{i-kt} \otimes x^p g^{i-(p+l)t} \right).  
\end{eqnarray*}
Comparing coefficients of $x^{k-(p+l)} g^{i-kt} \otimes x^p g^{i-(p+l)t}$ for any $0 \leq p \leq k-l$:
\begin{equation}\label{lambdaequality}
\lambda_j^i(k,p+l)\lambda_j^i(p+l,l)=\binom{k-l}{p}_{\xi^t} \xi^{-t(k-(p+l))p}\lambda_j^i(k,l).
\end{equation}
This establishes (2).
\end{proof}

Setting $l=0$ in Equation \eqref{lambdaequality} yields 
\begin{equation*}
\lambda_j^i(k,p)=\binom{k}{p}_{\xi^t} \xi^{-t(k-p)p} \dfrac{\lambda_j^i(k,0)}{\lambda_j^i(p,0)}   .
\end{equation*}
This explicit decomposition motivates the following  expression for $\lambda_j^i(k,p)$, where we adopt the convention that $\prod_{l=p+1}^{k} R_j^{i}(l,0) = 1$ when $p = k$ (empty product).

\begin{proposition}\label{prop:explictlambda}
 For all  $i,j\in \Z$ and $0\leq p\leq k  $, 
\begin{equation}\label{explictlambda}
	\lambda_j^i (k,p) = \binom{k}{p}_{\xi^t} \xi^{-t(k-p)p} \prod\limits_{l=p+1}^{k} R_j^{i}(l,0).
\end{equation}
\end{proposition}

\begin{proof}
We proceed by induction on $k$. The base case $k=0$ holds by definition. 
Assume \eqref{explictlambda} holds for all $0 \leq p \leq k$. For $k+1$ and $0 \leq p \leq k+1$:	
\begin{itemize}
	\item When $p=0$ or $p=k+1$, \eqref{explictlambda} holds by definition.
	
	\item For $0 < p < k+1$, the recursive definition gives:

\end{itemize}	
\begin{eqnarray*}
	& &\lambda_j^{i}(k+1,p)   \\[5pt]
	&=&R_j^{i}(k+1,p)\lambda_j^{i}(k,p)+\lambda_j^{i}(k,p-1)   \\[5pt]
	&=&(\xi^{t(j-p)}-\xi^{tk-i})\binom{k}{p}_{\xi^t} \xi^{-t(k-p)p} \prod\limits_{l=p+1}^{k}R_j^{i}(l,0) +\binom{k}{p-1}_{\xi^t} \xi^{-t(k-p+1)(p-1)} \prod\limits_{l=p}^{k}R_j^{i}(l,0)  \\[5pt]
	&=&\xi^{-t(k+1-p)p} \left(\prod\limits_{l=p+1}^{k} R_j^{i}(l,0)  \right)  \left( (\xi^{t(j-p)}-\xi^{tk-i}) \xi^{tp} \binom{k}{p}_{\xi^t}   +  \right.  \\[5pt]
	& &  \left.  (\xi^{tj}-\xi^{t(p-1)-i}) \xi^{t(k+1-p)} \binom{k}{p-1}_{\xi^t}  \right)
	\\[5pt]
	&=&\xi^{-t(k+1-p)p} \left(\prod\limits_{l=p+1}^{k} R_j^{i}(l,0)  \right) \left(\xi^{tj} \left(\binom{k}{p}_{\xi^t}+ \xi^{t(k+1-p)}\binom{k}{p-1}_{\xi^t} \right) - \right.    \\[5pt]
	& &  \left.  \xi^{tk-i} \left(\xi^{tp}\binom{k}{p}_{\xi^t}+ \binom{k}{p-1}_{\xi^t}  \right)   \right)
	\\[5pt]
	&=& \xi^{-t(k+1-p)p}  \left(  \prod\limits_{l=p+1}^{k} R_j^{i}(l,0) \right)  \left(\xi^{tj}  \binom{k+1}{p}_{\xi^t}- \xi^{tk-i} \binom{k+1}{p}_{\xi^t} \right)
	\\[5pt]
	&=& \binom{k+1}{p}_{\xi^t} \xi^{-t(k+1-p)p} \prod\limits_{l=p+1}^{k+1} R_j^{i}(l,0).
\end{eqnarray*}
This completes the induction.	
\end{proof}
\begin{remark}
{\normalfont 
The coefficients $ \lambda_j^i (k,p) $ are closely related to the coefficients $ \beta_{k-p,k}^{i,j} $ introduced in  
\cite{MBG21}. For details, see the proof of Lemma \ref{lem:isoasbraidedspace} of this paper.}
\end{remark}
We now establish that $ \A_j^{i}(m) $ admits a comatrix structure.
\begin{corollary}\label{cor:Aijisacomatrix}
For all $ i,j\in \Z $ and $ m\geq 0 $, $ \A_j^{i}(m) $ is a comatrix.
\end{corollary}

\begin{proof}
Fix $0 \leq l \leq k \leq m$ and $0 \leq p \leq k-l$. By Proposition \ref{prop:explictlambda},
\begin{eqnarray*}
	& &	\lambda_j^i(k,p+l)\lambda_j^i(p+l,l)    \\[5pt]
	&=& \left(\binom{k}{p+l}_{\xi^t} \xi^{-t(k-(p+l))(p+l)} \prod\limits_{s=p+l+1}^{k} R_j^{i}(s,0) \right)  \left(\binom{p+l}{l}_{\xi^t} \xi^{-tpl} \prod\limits_{s=l+1}^{p+l} R_j^{i}(s,0) \right)  \\[5pt]
	&=&  \left(\binom{k}{p+l}_{\xi^t} \binom{p+l}{l}_{\xi^t} \right) \xi^{-t(k-(p+l))(p+l)} \xi^{-tpl} \prod\limits_{s=l+1}^{k} R_j^{i}(s,0) \\[5pt]
	&=& \left(\binom{k-l}{p}_{\xi^t} \binom{k}{l}_{\xi^t} \right) \xi^{-t(k-(p+l))p} \xi^{-t(k-l)l} \prod\limits_{s=l+1}^{k} R_j^{i}(s,0)
	\\[5pt]
	&=& \binom{k-l}{p}_{\xi^t} \xi^{-t(k-(p+l))p}\lambda_j^i(k,l).
\end{eqnarray*}
Therefore, $\mathcal{A}_j^{i}(m)$ satisfies the comatrix criterion in Lemma \ref{lem:comatrixcriteration}.
\end{proof}

\subsection{Further Analysis of the Comatrix}\label{subsection3}
In this subsection, we give a deeper analysis of the comatrices defined in the preceding subsection, thereby obtaining key properties of finite-dimensional simple Yetter-Drinfeld modules over $ H $.

First, we obtain the following result.
\begin{proposition}\label{pro:ctijequal0}
Let $V$ be a simple Yetter-Drinfeld module over $H$ with $\dim_{\k} V = p+1$ for $p \geq 0$.  If $V_{[j]}^{[i]} \neq 0$ for $i,j \in \mathbb{Z}$, then:
\begin{itemize}
	\item [(1)] $ c_j^i(p+1,l)=0 $ for all $ 0\leq l \leq p $;
	\item [(2)] $ i \equiv ti_1 \pmod n  $ for some $ i_1\in \Z $.
\end{itemize}
\end{proposition}

\begin{proof}
Fix a non-zero element $v \in V_{[j]}^{[i]}$. By Lemma \ref{thm:standbasis}, $\{v, x \cdot v, \dots, x^p \cdot v\}$ constitutes a basis of $V$. Thus
\[
x^{p+1} \cdot v = \sum_{k=0}^{p} a_k x^k \cdot v \quad \text{for some } a_k \in \k.
\]
Compute the coaction:
\begin{eqnarray*}
   \delta (x^{p+1}\cdot v) 
&=&\delta \left(\sum\limits_{k=0}^p a_k  x^{k}\cdot v \right) \\
&=&\sum\limits_{k=0}^p \sum\limits_{l=0}^k a_k c_j^i(k,l) \otimes x^l \cdot v\\
&=&\sum\limits_{l=0}^p \left(\sum\limits_{k=l}^p a_k c_j^i(k,l) \right) \otimes   x^l \cdot v.
\end{eqnarray*}
Alternatively:
\begin{eqnarray*}
\delta (x^{p+1}\cdot v)
&=& \sum\limits_{l=0}^{p+1} c_j^i(p+1,l) \otimes x^l\cdot v  \\
&=&\sum\limits_{l=0}^{p} c_j^i(p+1,l) \otimes x^l\cdot v + c_j^i(p+1,p+1) \otimes x^{p+1}\cdot v  \\
&=&\sum\limits_{l=0}^{p} \left(c_j^i(p+1,l)+a_lc_j^i(p+1,p+1) \right)  \otimes x^l\cdot v.
\end{eqnarray*}
Comparing coefficients for $0 \leq l \leq p$:
\begin{equation}\label{eq:coeffeq}
\sum_{k=l}^p a_k c_j^i(k,l) = c_j^i(p+1,l) + a_l c_j^i(p+1,p+1).
\end{equation}
By Proposition \ref{pro:cijformula}, $ c_j^i(k,l)=\lambda_j^i(k,l)x^{k-l}g^{i-kt} $. Substituting into (\ref{eq:coeffeq}):
\[
\lambda_j^i(p+1,l) x^{p+1-l} g^{i-(p+1)t} = \sum_{k=l}^p a_k \lambda_j^i(k,l) x^{k-l} g^{i-kt} - a_l g^{i-(p+1)t}.
\]
Comparing exponents of $x$, we obtain $\lambda_j^i(p+1,l) = 0$, hence $c_j^i(p+1,l) = 0$. This proves (1).
	
For (2), note $\lambda_j^i(p+1,0) = \prod_{k=1}^{p+1} R_j^i(k,0) = 0$. Thus $R_j^i(k,0) = \xi^{tj} - \xi^{t(k-1)-i} = 0$ for some $1 \leq k \leq p+1$. Therefore,
\[
i \equiv t(k-1-j) \pmod{n},
\]
which gives (2) with $i_1 = k-1-j$.
\end{proof}

\begin{remark}\label{rmk:onlyti}
{\normalfont 
By Proposition \ref{pro:ctijequal0}, for a finite-dimensional simple Yetter-Drinfeld module $V$ with standard element $v$, we have $v \in V_{[j]}^{[ti]}$ for some $i,j \in \mathbb{Z}$. The converse also holds by Corollary \ref{cor:coversectijequal0}.}
\end{remark}

Now, we restrict our attention to matrices of the form $\mathcal{A}_j^{ti}(p)$. Recall that $ w=\xi^t $, which implies $w^N = 1$ and $R_j^{ti}(k,l) = w^{j-l} - w^{k-1-i}$. For any integer $k \geq 1$,
\begin{equation*}
	c_j^{ti}(k,0) = \lambda_j^{ti}(k,0) x^k g^{t(i-k)} = \left(\prod\limits_{l=1}^{k}(w^{j}-w^{l-1-i}  ) \right) x^k g^{t(i-k)}.
\end{equation*}
Since this expression vanishes for some $k \geq 1$, we introduce the following definition.

\begin{definition}
For any $ i,j\in \Z $ , define $ \overline{\A_j^{ti}}:=\A_j^{ti}(m) $, where $ m\geq 0 $ is the smallest integer such that $ c_j^{ti}(m+1,0)=0 $.
\end{definition}
The following properties are immediate from the definitions.
\begin{remark}\label{rmk:Atijblock}
\hfill
{\normalfont 
\begin{itemize}
	\item [(1)] As established, $c_j^{ti}(m+1,0) = 0$ if and only if $\lambda_j^{ti}(m+1,0) = 0$.
	Therefore, $m$ is also the smallest integer such that $\lambda_j^{ti}(m+1,0) = 0$, 
	and hence the smallest integer such that $R_j^{ti}(m+1,0) = 0$. 
	\item [(2)] Proposition \ref{prop:explictlambda} implies $\lambda_j^{ti}(k,l) \neq 0$ for all $0 \leq l \leq k \leq m$,  meaning $\overline{\mathcal{A}_j^{ti}}$ has strictly non-zero entries throughout its lower triangular part.
	\item [(3)]  Proposition \ref{prop:explictlambda} implies $\lambda_j^{ti}(k,l) = 0$ for all $0 \leq l \leq m <k $,  so all entries below $\overline{\mathcal{A}_j^{ti}}$ vanish.
\end{itemize}}
\end{remark}

Motivated by these observations, we obtain the following matrix decomposition.

\begin{lemma}\label{lem:comatrixdecomposition}
For $i,j \in \mathbb{Z}$, assume $\overline{\mathcal{A}_j^{ti}} = \mathcal{A}_j^{ti}(m)$. Then for any integer $p > m$,
\begin{equation*}
	\A_j^{ti}(p)=
	\left(\begin{array}{cccc}
		\A_j^{ti}(m) &   0           \\[10pt]
		0            &  \A_{j-(m+1)}^{t(i-(m+1))}(p-m-1)
	\end{array}
	\right).
\end{equation*}
\end{lemma}

\begin{proof}
Let $m < k \leq p$. To establish the decomposition, it suffices to prove	
\begin{equation}\label{ctiandcti-mt}
	\lambda_j^{ti}(k,l)=
	\begin{cases}
		0, & 0\leq l<m+1 ,\\[5pt]
		\lambda_{j-(m+1)}^{t(i-(m+1))}(k-(m+1),l-(m+1)) , \qquad  & m+1\leq l \leq k .
	\end{cases}
\end{equation} 
for all such $k$.

\noindent \textbf{Case 1:} $0 \leq l < m+1$.
This follows immediately from Remark \ref{rmk:Atijblock}(3), which gives $\lambda_j^{ti}(k,l) = 0$.

\noindent \textbf{Case 2:} $m+1 \leq l \leq k$. We proceed by induction on $k$. The base case $k = m+1$ holds trivially. Now assume the statement holds for some $k$ with $m < k < p$. For $k+1$, we consider three possibilities:
\begin{itemize}
	\item [(1)] $l = k+1$: The result holds by definition.
	\item [(2)] $m+1 < l < k+1$: By the induction hypothesis,
\begin{eqnarray*}
	& &	\lambda_j^{ti}(k+1,l) \\[5pt]
	&=& R_j^{ti}(k+1,l)\lambda_j^{ti}(k,l)+\lambda_j^{ti}(k,l-1)  \\[5pt]
	&=& R_{j-(m+1)}^{t(i-(m+1))}(k+1-(m+1),l-(m+1))\lambda_{j-(m+1)}^{t(i-(m+1))}(k-(m+1),l-(m+1)) \\[5pt] 
	& & + \lambda_{j-(m+1)}^{t(i-(m+1))}(k-(m+1),l-1-(m+1)) 
	\\[5pt]
	&=& \lambda_{j-(m+1)}^{t(i-(m+1))}(k+1-(m+1),l-(m+1)).
\end{eqnarray*}
   \item [(3)] $l = m+1$: The result follows similarly.
\end{itemize}
Thus, the induction extends to $k+1$, completing the proof.
\end{proof}

We now demonstrate that the matrix decomposition exhibits a periodic structure. To establish this, we first define a function $\phi$ that provides an explicit expression for the size of $  \overline{\A_j^{ti}} $.  For any  $i \in \mathbb{Z}$, let $\phi(i)$ denote the unique integer in $\{1, 2, \dots, N\}$ satisfying $i \equiv \phi(i) \pmod{N}$.

\begin{proposition}\label{lem:Atijblocks}
For any  $ i,j\in \Z $ , assume  $ \overline{\A_j^{ti}}=\A_j^{ti}(m) $. Then for any integer $p\geq 1$,
\begin{itemize}
	\item [(1)] $ m=N- \phi(-i-j) $.
	\item [(2)] The $N \times N$ matrix decomposes as
	\begin{equation}\label{Atijblock}
		\A_j^{ti}(N-1)=
		\begin{cases}
			\hspace{3.2ex}	\overline{\A_j^{ti}}, & m=N-1  , \\[10pt]
			\left(\begin{array}{cccc}
			\overline{\A_j^{ti}} &   0           \\[10pt]
			0           & \overline{\A_{j-(m+1)}^{t(i-(m+1))}} 
				 \end{array}
				\right)  , \qquad & 0\leq m <N-1 .
		\end{cases}
	\end{equation}
	\item [(3)] For multiples of $N$, the matrix has block-diagonal structure:
	\begin{equation}\label{Atijblocks}
		\A_j^{ti}(pN-1)=
		\left(\begin{array}{cccc}
		 \A_j^{ti}(N-1) &              \\[10pt]
				&  \ddots  &      \\[10pt]
				&  &         \A_j^{ti}(N-1)
		\end{array}
		\right)
	\end{equation}
    containing $p$ identical blocks along the diagonal.
\end{itemize}
\end{proposition}

\begin{proof}
For any $k \geq 1$, we have 
\[
R_j^{ti}(k,0) = w^j - w^{k-1-i} = w^j(1 - w^{k-1-i-j}).
\]
Observe that $N - \phi(-i-j)$ is the smallest integer satisfying $R_j^{ti}(N - \phi(-i-j) + 1, 0) = 0$. 
Thus $m = N - \phi(-i-j)$, establishing part (1).	
	
For part (2), we consider only $0 \leq m < N-1$. By Lemma \ref{lem:comatrixdecomposition}, it suffices to verify 	$\overline{\mathcal{A}_{j-(m+1)}^{t(i-(m+1))}} = \mathcal{A}_{j-(m+1)}^{t(i-(m+1))}(N-m-2) $. Assuming $\overline{\mathcal{A}_{j-(m+1)}^{t(i-(m+1))}} = \mathcal{A}_{j-(m+1)}^{t(i-(m+1))}(k)$, part (1) yields	
\begin{eqnarray*}
   k&=&	N-\phi\left(-(i-(m+1))-(j-(m+1))\right) \\[5pt]
	&=& N-\phi\left(-i-j+2m+2)\right) \\[5pt]
	&=& N-\phi\left(-i-j+2(N-\phi(-i-j))+2 \right) \\[5pt] 
	&=& N-\phi\left(-i-j-2\phi(-i-j)+2 \right) \\[5pt]
	&=& N-\phi\left(-\phi(-i-j)+2 \right).
\end{eqnarray*}
Since $0 \leq m = N - \phi(-i-j) < N-1$, we have $-N+2 \leq -\phi(-i-j) + 2 \leq 0$, and
\begin{eqnarray*}
  k &=& N-\phi\left(-\phi(-i-j)+2 \right)  \\[5pt]
    &=&  N-(-\phi(-i-j)+2+N)   \\[5pt]
    &=&  \phi(-i-j)-2    \\[5pt]
    &=&  N-m-2,
\end{eqnarray*}
concluding part (2).

For part (3), we proceed by induction on $ p $. The case $p=1$ holds trivially. Assume (\ref{Atijblocks}) holds for some $p$. By the property $\mathcal{A}_j^{ti}(k) = \mathcal{A}_{j_1}^{ti_1}(k)$ whenever $i \equiv i_1 \pmod{N}$ and $j \equiv j_1 \pmod{N}$), Lemma \ref{lem:comatrixdecomposition}, and part (2) yield
\begin{eqnarray*}
\A_j^{ti}((p+1)N-1)
    &=& \left(\begin{array}{cccc}
     	      \A_j^{ti}(N-1) &   0           \\[10pt]
    	      0           & \A_{j-N}^{t(i-N)}(pN-1)
              \end{array}
              \right)  \\[5pt]
	&=&  \left(\begin{array}{cccc}
		      \A_j^{ti}(N-1) &   0           \\[10pt]
		      0           & \A_j^{ti}(pN-1)
	           \end{array}
	     \right)  \\[5pt]
	&=&  \left(\begin{array}{cccc}
		  \A_j^{ti}(N-1) &              \\[10pt]
		  &  \ddots  &      \\[10pt]
		  &  &         \A_j^{ti}(N-1)
	           \end{array}
	     \right),
\end{eqnarray*}
completing the induction step and the proof.
\end{proof}

With the decomposition of $\mathcal{A}_j^{ti}(p)$ established, we now present two key consequences. The first states that any finite-dimensional simple module has dimension at most $n$.

\begin{corollary}\label{cor:dimensionleqn}
Let $V$ be a Yetter-Drinfeld module over $H$, and let $i, j \in \mathbb{Z}$. Then:
\begin{itemize}
	\item [(1)] $ x^n \cdot (V_{[j]}^{[ti]}) \subseteq V_{[j]}^{[ti]} $.
	\item [(2)] If $V$ is finite-dimensional and simple, then $\dim_{\k} V \leq n$.
\end{itemize}

\end{corollary}

\begin{proof}
Let $v \in V_{[j]}^{[ti]}$. By Proposition \ref{pro:spaniscomodule}(1), we have $x^n \cdot v \in V_{[j-n]} = V_{[j]}$. Furthermore, Proposition \ref{lem:Atijblocks}(3) implies $c_j^{ti}(n, l) = 0$ for all $0 \leq l < n$, and hence
$$ \delta(x^n\cdot v)=\sum\limits_{l=0}^{n} c_j^{ti}(n,l) \otimes x^l \cdot v=c_j^{ti}(n,n)\otimes x^n\cdot v=g^{t(i-n)}\otimes x^n\cdot v=g^{ti}\otimes x^n\cdot v. $$
This shows $x^n \cdot v \in V^{[ti]}$, establishing part (1).

For part (2), let $u$ be a standard element of $V$. By Remark \ref{rmk:onlyti}, $u \in V_{[j]}^{[ti]}$ for some $i, j \in \mathbb{Z}$, so in particular $V_{[j]}^{[ti]} \neq 0$. By part (1), $V_{[j]}^{[ti]}$ is $x^n$-invariant. Since $\Bbbk$ is algebraically closed, the operator $x^n$ has an eigenvector $v \in V_{[j]}^{[ti]}$ with eigenvalue $\lambda \in \Bbbk$, i.e., $x^n \cdot v = \lambda v$. Then the set $\{v, x \cdot v, \dots, x^n \cdot v = \lambda v\}$ is linearly dependent. By Lemma \ref{lem:standbasis}, it follows that $\dim_{\Bbbk} V \leq n$, completing the proof.	
\end{proof}

The second result provides the converse of Proposition \ref{pro:ctijequal0}(2).

\begin{corollary}\label{cor:coversectijequal0}
Let $V$ be a simple Yetter-Drinfeld module over $H$. If  $ V_{[j]}^{[ti]}\neq 0$ for $i, j \in \mathbb{Z}$, then $V$ is finite-dimensional.
\end{corollary}

\begin{proof}
Suppose, for contradiction, that $V$ is infinite-dimensional. Fix a nonzero element $v \in V_{[j]}^{[ti]}$. By Lemma \ref{thm:standbasisinfinite}, the set $\{x^k \cdot v \mid k \geq 0\}$ forms a basis for $V$. 

Moreover, Corollary \ref{cor:dimensionleqn}(1) implies $x^n \cdot v \in V_{[j]}^{[ti]}$. Applying Lemma \ref{thm:standbasisinfinite} again, the set $\{x^k \cdot v \mid k \geq n\}$ must also form a basis of $V$. 
However, this leads to a contradiction since $\{x^k \cdot v \mid k \geq n\}$ is a proper subset of $\{x^k \cdot v \mid k \geq 0\}$. 
\end{proof}

\section{Classification of Simple Yetter-Drinfeld Modules over $ H $ }\label{section4}

In this section, we provide a complete classification of the simple Yetter-Drinfeld modules over $H$, explicitly describing their structure and determining all isomorphism relations between them. Subsection~\ref{subsection4.1} focuses on the finite-dimensional case, while Subsection~\ref{subsection4.2} is devoted to the infinite-dimensional simple modules. This classification yields a full description of all simple objects in the Yetter-Drinfeld category over $H$.

\subsection{The Finite-Dimensional Case}\label{subsection4.1}

For any integers $i, j \in \mathbb{Z}$ and scalar $\lambda \in \Bbbk$, we define a finite-dimensional simple Yetter-Drinfeld module $V(ti,j,\lambda)$. Recall from Proposition~\ref{lem:Atijblocks} that $\overline{\mathcal{A}_j^{ti}} = \mathcal{A}_j^{ti}(N - \phi(-i-j))$.

The module $V(ti,j,\lambda)$ is defined as follows:

\begin{itemize}
	\item $\left\{v_0,v_1,\cdots,v_m \right\}  $ is a $ \k $-basis of  $ V(ti,j,\lambda) $, where 
	\begin{equation*}
		m=
		\begin{cases}
			N-\phi(-i-j), & \lambda=0 , \\[5pt]
			n-1 ,         & \lambda \neq 0  .
		\end{cases}
	\end{equation*}
	\item The action of $g$ is defined for $0 \leq k \leq m$ by
	\begin{equation}\label{basisofV}
		g\cdot v_k=\xi^{j-k} v_k;
	\end{equation}
	\item The action of $x$ is defined for $0 \leq k \leq m$ by
	\begin{equation*}
		x\cdot v_k=
		\begin{cases}
			v_{k+1} , & 0\leq k<m  ,\\[5pt]
			\lambda v_0    ,   & k=m  .
		\end{cases}
	\end{equation*}
	\item The coaction is defined for $0 \leq k \leq m$ by
	\begin{equation*}
		\delta(v_k) =\sum\limits_{l=0}^{k} c_j^{ti}(k,l) \otimes v_l.
	\end{equation*}
\end{itemize}

\begin{remark}\label{rmk:Vtijlambda}
{\normalfont Denote $V(ti,j,\lambda)$ by $V$.
\begin{itemize}
	\item [(1)] The indices $i$ and $j$ indicate that $V$ has a standard element of type $(ti,j)$.
	\item [(2)] The scalar $\lambda$ determines the action of $x^n$ on $V$, satisfying $x^n = \lambda \operatorname{id}_V$, i.e., $x^n \cdot v = \lambda v$ for all $v \in V$. Hence, if $\lambda \neq \mu$, then $V(*,*,\lambda)$ is not isomorphic to $V(*,*,\mu)$.
	\item [(3)] The braiding induced by $ V $ is of triangular type.
\end{itemize}}
\end{remark}

\begin{remark}\label{rmk:cannotrealize}
{\normalfont 
Let $ q\in \k $ be an $ m $-th root of unity of order $ d $ and $ \mu \in \k $. In \cite{R75} and \cite{AS98}, Radford and Andruskiewitsch-Schneider have considered the following Hopf algebra $ A(m,d,\mu,q) $, which as an associative algebra is generated by $ g $ and $ x $ with relations
\begin{equation*}
	g^m=1, \quad x^d=\mu(1-g^d), \quad xg=qgx.
\end{equation*}
    Its comultiplication $ \Delta $, counit $ \varepsilon $, and the antipode $ S $ are given by
\begin{equation*}
	\Delta(g)=g\otimes g, \ \varepsilon(g)=1, \ \Delta(x)=x\otimes 1 + g \otimes x, \  \varepsilon(x)=0, \ S(g)=g^{-1}, \ S(x)=-g^{-1}x.
\end{equation*}
The algebra $ A(m,d,\mu,q) $ reduces to the generalized Taft algebra \cite{HCZ04} when $ \mu=0 $  and to the $ m^2 $-dimensional Taft algebra \cite{T71} when $ d=m $.

When $ t $ and $ n $ are coprime,  the module $ V(ti,j,0) $ is in fact a simple Yetter-Drinfeld module over the finite-dimensional Taft algebra $ H(n,t,\xi,0) $ (see Remark \ref{rmk:finiteTaft}(2)), with the same action and coaction structure. Such modules have been classified independently in \cite{C00} and \cite{MBG21}.

However, if $ \lambda \neq 0 $, the module $ V(ti,j,\lambda) $  cannot be realized as a simple Yetter-Drinfeld module over any algebra of the form  $ A(m,d,\mu,q) $. The argument is as follows.

Suppose, for contradiction, that  $ V(ti,j,\lambda) $   is such a module over some  $ A(m,d,\mu,q) $. 
By an argument analogous to Lemma~\ref{lem:standbasis}, any simple Yetter-Drinfeld module $ W $ over  $ A(m,d,\mu,q) $ that admits a standard basis (as in Subsection~\ref{subsection31}) must satisfy $ \dim_{\k} W \leq d $. Moreover, if $ \dim_{\k} W < d $, then $ x^d \cdot v=0 $ for all $ v\in W  $. Consequently, for $ V(ti,j,\lambda) $ we must have $ \dim_{\k}  V(ti,j,\lambda) =d $, which forces $ d=n $.

Now, let $\{v_0, v_1, \dots, v_{n-1}\}$ be the basis of $V(ti,j,\lambda)$ defined in \eqref{basisofV}. Then for any $ 0\leq k \leq n-1 $,
\begin{equation*}
	x^d \cdot v_k =\mu (1-g^d)\cdot v_k=\mu (1-g^n)\cdot v_k=\mu (1-\xi^{(j-k)n})  v_k=0.
\end{equation*}
This implies $ x^n \cdot v=0 $ for all $ v\in V(ti,j,\lambda)  $, contradicting Remark~\ref{rmk:Vtijlambda}(2).	}
\end{remark}

We first prove that the modules constructed above are indeed simple Yetter-Drinfeld modules.
\begin{lemma}\label{lem:yanzhengYDmo}
For any $ i,j\in \Z $ and $ \lambda \in \k $, the module $ V(ti,j,\lambda) $  is a simple Yetter-Drinfeld module over $ H $. 
\end{lemma}
\begin{proof}
Let $\{v_0, v_1, \dots, v_m\}$ be the basis of $V(ti,j,\lambda)$ defined in \eqref{basisofV}, and denote $V(ti,j,\lambda)$ by $V$. We first show that $V$ is a Yetter-Drinfeld module.

By Lemma~\ref{lem:comoduleequicodition} and Corollary~\ref{cor:Aijisacomatrix}, $V$ is a left $H$-comodule. Clearly, $V$ is also a left $\Bbbk G$-module, and $\operatorname{span}\{v_k\} = V_{[j-k]}$ for all $0 \leq k \leq m$. By Proposition~\ref{prop:structureofHmodule}, $H$ is a left $H$-module.	
It remains to verify the compatibility condition \eqref{compatibilitycondition} for all $h \in H$ and $v \in V$. Define
\[
\rho(h \cdot v) = h_{(1)} v_{(-1)} S(h_{(3)}) \otimes h_{(2)} \cdot v_{(0)}.
\]	
By Lemma~\ref{lem:compatiable}, it is enough to show that for all $0 \leq k \leq m$,	
\begin{eqnarray}
	\delta(g \cdot v_k) &= \rho(g \cdot v_k), \label{gcompatib} \\
	\delta(x \cdot v_k) &= \rho(x \cdot v_k). \label{xcompatib}
\end{eqnarray}

For~\eqref{gcompatib}, we have
\begin{eqnarray*}
\rho (g\cdot v_k) 
 &=&\sum\limits_{l=0}^{k} gc_j^{ti}(k,l) g^{-1} \otimes g \cdot v_l    \\
 &=&\sum\limits_{l=0}^{k} g \left(\lambda_j^{ti}(k,l) x^{k-l}g^{i-kt} \right) g^{-1} \otimes g \cdot v_l   \\
 &=&\sum\limits_{l=0}^{k} \xi^{-(k-l)} \lambda_j^{ti}(k,l) x^{k-l}g^{i-kt}  \otimes \xi^{j-l} v_l    \\
 &=&\xi^{j-k} \sum\limits_{l=0}^{k}  c_j^{ti}(k,l)   \otimes  v_l    \\[3pt]
 &=&\delta(g\cdot v_k),
\end{eqnarray*}
so the condition holds.

To verify~\eqref{xcompatib}, denote $x \cdot v_m$ by $v_{m+1}$. Using the definition of $c_j^{ti}(k,l)$, we compute
\begin{eqnarray*}
	& &\rho(x \cdot v_k) \\
	&=&\sum\limits_{l=0}^k \left(c_j^{ti}(k,l)S(x)\otimes 1\cdot v_l+ c_j^{ti}(k,l)S(g^t)\otimes x\cdot v_l+xc_j^{ti}(k,l)S(g^t)\otimes g^t\cdot v_l \right) \\
	&=&\sum\limits_{l=0}^{k+1} c_j^{ti}(k+1,l) \otimes v_l
\end{eqnarray*}
for all $0 \leq k \leq m$. This establishes~\eqref{xcompatib} for $0 \leq k < m$.

For $k = m$, note that $c_j^{ti}(m+1, l) = 0$ for all $0 \leq l \leq m$ by Remark~\ref{rmk:Atijblock}(3) (if $m = N - \phi(-i-j)$) or by Proposition~\ref{lem:Atijblocks}(3) (if $m = n-1$). Then
\begin{eqnarray*}
       \rho(x \cdot v_m)
    &=&\sum\limits_{l=0}^{m+1} c_j^{ti}(m+1,l) \otimes v_l \\
	&=&g^{ti-t(m+1)}\otimes v_{m+1} \\[6pt]
	&=&\begin{cases}
		0    &   \lambda =0  \\
		g^{ti}\otimes \lambda v  & \lambda \neq 0 
	\end{cases},
\end{eqnarray*}
and it follows that $\delta(x \cdot v_m) = \rho(x \cdot v_m)$. Therefore, the compatibility condition holds, and $V$ is a Yetter-Drinfeld module.

We now prove that $V$ is simple. Let $W$ be a nonzero Yetter-Drinfeld submodule of $V$; we show that $W = V$. Choose $0 \neq v \in W$. By Lemma~\ref{lem:kGmoduleproperty}, we may assume $v = v_k$ for some $0 \leq k \leq m$.

\noindent \textbf{Case 1: $\lambda = 0$.} Then $\mathcal{A}_j^{ti}(m) = \overline{\mathcal{A}_j^{ti}}$. Since $W$ is a subcomodule,
\[
\delta(v_k) = \sum_{l=0}^k c_j^{ti}(k, l) \otimes v_l \in H \otimes W.
\]
By Remark~\ref{rmk:Atijblock}(2), $\lambda_j^{ti}(k, l) \neq 0$ for all $0 \leq l \leq k$, so the set $\{c_j^{ti}(k, 0), c_j^{ti}(k, 1), \dots, c_j^{ti}(k, k)\}$ is linearly independent. Hence, $v_0 \in W$, and thus $v_l = x^l \cdot v_0 \in W$ for all $0 \leq l \leq m$. Therefore, $V = W$.

\noindent \textbf{Case 2: $\lambda \neq 0$.} Then
\[
v_0 = \frac{1}{\lambda} x^{n-k} \cdot v_k \in W,
\]
so again $V = W$.

This completes the proof.
\end{proof}

We next prove that the modules constructed above exhaust all finite-dimensional simple Yetter-Drinfeld modules.

\begin{proposition}\label{prop:finiteisotoV}
Every finite-dimensional simple Yetter-Drinfeld module $V$ over $H$ is isomorphic to $V(ti,j,\lambda)$ for some  $i, j \in \mathbb{Z}$ and  $\lambda \in \k$.	
\end{proposition}

\begin{proof}
Assume $\dim_{\Bbbk} V = p + 1$. By Corollary~\ref{cor:dimensionleqn}(2), we have $p \leq n - 1$. Let $v$ be a standard element of $V$; by Remark~\ref{rmk:onlyti}, $v \in V_{[j]}^{[ti]}$ for some $i, j \in \mathbb{Z}$. Lemma~\ref{thm:standbasis} implies that $\{v, x \cdot v, \dots, x^p \cdot v\}$ is a basis of $V$.

\noindent\textbf{Case 1:} $ x^{p+1} \cdot v=0 $.  
Let $m = N - \phi(-i-j)$. By Proposition~\ref{lem:Atijblocks}, we have $\overline{\mathcal{A}_j^{ti}} = \mathcal{A}_j^{ti}(m)$. We claim that $p = m$.

First, suppose $p < m$. Then Remark~\ref{rmk:Atijblock}(2) implies $c_j^{ti}(p+1, 0) \neq 0$, contradicting Proposition~\ref{pro:ctijequal0}. Hence, $p \geq m$.
Now suppose $p > m$. By Proposition~\ref{lem:Atijblocks}, $x^{m+1} \cdot v$ is also a standard element of $V$. Then the subspace $\operatorname{span}\{x^{m+1} \cdot v, \dots, x^p \cdot v\}$ is a Yetter-Drinfeld submodule of $V$, contradicting the simplicity of $V$. Therefore, $p = m$.

Let $\{v_0, v_1, \dots, v_m\}$ be the basis of $V(ti, j, 0)$ defined in~\eqref{basisofV}. Define a linear map $f: V(ti, j, 0) \to V$ by $f(v_k) = x^k \cdot v$ for all $0 \leq k \leq m$. It is straightforward to verify that $f$ is an isomorphism of Yetter-Drinfeld modules, so $V \cong V(ti, j, 0)$.

\noindent\textbf{Case 2:} $ x^{p+1} \cdot v\neq 0 $.  Note that
\begin{equation*}
V=\bigoplus\limits_{k=0}^{p} \span \left\{x^k \cdot v\right\}=\bigoplus\limits_{k=0}^{p} V_{[j-k]} .
\end{equation*}
If $p < n - 1$, then $V_{[j-(p+1)]} = 0$ because $[j - (p+1)] \neq [j - k]$ for all $0 \leq k \leq p$. This implies $x^{p+1} \cdot v = 0$, a contradiction. Therefore, $p = n - 1$ and
\[
x^{p+1} \cdot v = x^n \cdot v \in V_{[j - n]} = V_{[j]} = \operatorname{span}\{v\}.
\]
Thus, $x^n \cdot v = \lambda v$ for some nonzero $\lambda \in \Bbbk$.

Let $\{v_0, v_1, \dots, v_{n-1}\}$ be the basis of $V(ti, j, \lambda)$ defined in~\eqref{basisofV}. Define a linear map $f: V(ti, j, \lambda) \to V$ by $f(v_k) = x^k \cdot v$ for all $0 \leq k \leq n - 1$. Again, $f$ is an isomorphism of Yetter-Drinfeld modules, completing the proof.
\end{proof}

We now analyze the isomorphism relations among these simple modules.
Two Yetter-Drinfeld modules $V$ and $W$ over $H$ are said to have standard elements of the same type if there exist integers $i, j \in \mathbb{Z}$ such that both $V_{[j]}^{[i]} \neq 0$ and $W_{[j]}^{[i]} \neq 0$.

The following lemma is immediate and states that each simple module is uniquely determined by the parameter $\lambda$ and the type of its standard elements.

\begin{lemma}\label{lem:isomorphism}
For any $i, j, r, s \in \mathbb{Z}$ and $\lambda, \mu \in \Bbbk$, the Yetter-Drinfeld modules $V(ti,j,\lambda)$ and $V(tr,s,\mu)$ are isomorphic if and only if they have  standard elements of the same type and $\lambda = \mu$.
\end{lemma}

\begin{proof}
Let $V = V(ti,j,\lambda)$ and $W = V(tr,s,\mu)$.

\noindent ($\Rightarrow$) Suppose $V \cong W$, and let $f: V \to W$ be an isomorphism. Since $f$ preserves both the action and coaction structures, $f(V_{[p]}^{[q]}) \subseteq W_{[p]}^{[q]}$ for all $p, q \in \mathbb{Z}$.  Hence, $V$ and $W$ have standard elements of the same type.

Now choose a nonzero element $v \in V$. By Remark~\ref{rmk:Vtijlambda}(2), we have:
\[
\lambda f(v) = f(x^n \cdot v) = x^n \cdot f(v) = \mu f(v).
\]
Since $f(v) \neq 0$, it follows that $\lambda = \mu$.	
	
\noindent ($\Leftarrow$) Conversely, suppose $V$ and $W$ have standard elements of the same type and $\lambda = \mu$. By Remark~\ref{rmk:onlyti}, there exist $p, q \in \mathbb{Z}$ such that $V_{[p]}^{[tq]} \neq 0$ and $W_{[p]}^{[tq]} \neq 0$. Since $V$ is simple, the proof of Proposition~\ref{prop:finiteisotoV} implies $V \cong V(tq, p, \lambda)$. Similarly, $W \cong V(tq, p, \lambda)$. Therefore, $V \cong W$, as required.	
\end{proof}

To clarify the isomorphism relations among the simple modules, it suffices to determine their standard elements. Let $V = V(ti,j,\lambda)$ with basis $\{v_0, v_1, \dots, v_m\}$ as defined in \eqref{basisofV}. By Lemma~\ref{lem:socV}(2), this reduces to computing the decomposition of $\soc(V)$. Note that
\begin{equation*}
V=\bigoplus\limits_{k=0}^{m} \span \left\{v_k\right\}=\bigoplus\limits_{k=0}^{m} V_{[j-k]} .
\end{equation*}
In particular, for any $1 \leq k \leq n$, we have $\dim_{\Bbbk} V_{[k]} \leq 1$. Therefore, if $V_{[p]}^{[q]} \neq 0$ for some $p, q \in \mathbb{Z}$, it follows that $V_{[p]}^{[q]} = V_{[p]}$. Consequently, we only need to find all $v_k$ such that $\operatorname{span}\{v_k\}$ is a subcomodule. This can be done by examining the comatrix $\mathcal{A}_j^{ti}(m)$; see Proposition~\ref{lem:Atijblocks} for details. We thus obtain the following result.

\begin{lemma}\label{lem:socVtij}
For any $ i,j\in \Z $ and $  \lambda \in \k $, let $p=N-\phi(-i-j)+1 $ and $ d=\frac{n}{N} $. Let $V = V(ti,j,\lambda)$ with basis $\{v_0, v_1, \dots, v_m\}$ as defined in \eqref{basisofV}. Then
\begin{equation*}
\soc (V)=
\begin{cases}
V_{[j]}^{[ti]},   &  \lambda=0 , \\[12pt]
\bigoplus\limits_{k=0}^{d-1} V_{[j-kN]}^{[ti]} , &  \lambda\neq 0 \  \operatorname{and} \ N-\phi(-i-j)=N-1 ,  \\[12pt]
\bigoplus\limits_{k=0}^{d-1}\left(V_{[j-kN]}^{[ti]} \bigoplus V_{[(j-p)-kN]}^{[t(i-p)]} \right) , \qquad &  \lambda\neq 0 \  \operatorname{and} \ N-\phi(-i-j)<N-1  .
\end{cases}
\end{equation*}
where  $ V_{[j-kN]}^{[ti]}=\span \left\{ v_{kN}\right\} $ and $ V_{[(j-p)-kN]}^{[t(i-p)]}=\span \left\{ v_{kN+p}\right\} $  for all $ 0\leq k \leq d-1 $.
\end{lemma}

\begin{proof}
We prove only the first case, as the proofs of the other two are analogous. As previously analyzed, it suffices to identify those elements $v_k$ for which $\operatorname{span}\{v_k\}$ is a subcomodule. Since $\mathcal{A}_j^{ti}(m) = \overline{\mathcal{A}_j^{ti}}$, Remark~\ref{rmk:Atijblock}(2) implies that only $v_0$ satisfies this condition. Hence, $\soc(V) = V_{[j]}^{[ti]}$.
\end{proof}

We thus arrive at a complete description of the isomorphisms among these simple modules.

\begin{proposition}\label{prop:isocondition}
For any $ i,j,r,s\in \Z $ and $ 0\neq \lambda \in \k $, let $p=N-\phi(-i-j)+1 $. 
\begin{itemize}
	\item [(1)]  $V(ti,j,0) \cong  V(tr,s,0) $ if and only if   $ i \equiv r  \pmod N  $ and  $ j \equiv s  \pmod n  $.
	\item [(2)] If $ p=N $, then $V(ti,j,\lambda) \cong V(tr,s,\lambda) $ if and only if   $ i \equiv r  \pmod N  $ and  $ j \equiv s  \pmod N  $.
	\item [(3)] If $ p<N $, then $V(ti,j,\lambda) \cong V(tr,s,\lambda) $ if and only if  either of the following holds:
	\begin{itemize}
		\item[(a)]  $ i \equiv r \pmod N   $ and  $ j \equiv s  \pmod N  $; or
		\item[(b)]  $ i \equiv r-p  \pmod N  $ and  $ j \equiv s-p  \pmod N  $.
	\end{itemize}
\end{itemize}
\end{proposition}

\begin{proof}
This follows immediately from Lemmas~\ref{lem:isomorphism} and \ref{lem:socVtij}.
\end{proof}

\subsection{The Infinite-Dimensional Case}\label{subsection4.2}

In this subsection, we address the classification of infinite-dimensional simple modules. We first note that by Corollary \ref{cor:coversectijequal0}, $ H $  can possess such modules only when  $ t $ and $ n $
are not coprime. We therefore make the standing assumption throughout this subsection that $ \gcd (t,n)>1 $, and proceed to prove that $ H $ indeed admits infinite-dimensional simple modules under this condition.

Define the following subeset of $ \Z $:
\begin{equation*}
	\mathcal{J}=\left\{i\in \Z \mid  i \not\equiv tk \pmod n \  \text{for any} \ k \in \Z\right\}.
\end{equation*}
Note that $\mathcal{J} \neq \emptyset$ since $t$ and $n$ are not coprime. For any $i \in \mathcal{J}$ and $j \in \mathbb{Z}$, we define an infinite-dimensional simple Yetter-Drinfeld module $V(i,j)$.

The module $V(i,j)$ is defined as follows:

\begin{itemize}
	\item A $ \k $-basis of  $ V(i,j) $ is given by
	\begin{equation*}
	  \left\{v_k \mid k \geq 0\right\};
	\end{equation*}
	\item The action of $g$ is defined for $k\geq 0$ by
	\begin{equation}\label{basisofVinfinite}
		g\cdot v_k=\xi^{j-k} v_k;
	\end{equation}
	\item The action of $x$ is defined for $k\geq 0$ by
	\begin{equation*}
		x\cdot v_k=v_{k+1};
	\end{equation*}
	\item The coaction is defined for $k\geq 0$ by
	\begin{equation*}
		\delta(v_k) =\sum\limits_{l=0}^{k} c_j^{i}(k,l) \otimes v_l.
	\end{equation*}
\end{itemize}

Similarly, we establish the following results.

\begin{lemma}\label{lem:yanzhengydmoinfinit}
For any $i \in \mathcal{J}$ and $j \in \mathbb{Z}$, the module $ V(i,j) $  is a simple Yetter-Drinfeld module over $ H $. 	
\end{lemma}
\begin{proof}
Let $ \left\{v_k \mid k \geq 0\right\}$ be the basis of $V(i,j)$ defined in \eqref{basisofVinfinite}, and denote $V(i,j)$ by $V$.  Similar to the proof of Lemma~\ref{lem:yanzhengYDmo}, one can verify that $V$ is a Yetter-Drinfeld module. We therefore focus on proving that $V$ is simple.

Let $W$ be a non-zero Yetter-Drinfeld submodule of $V$; we show that $W = V$. Let $U$ be a simple subcomodule of $W$. Since $H$ is pointed, $U = \operatorname{span}\{v\}$ for some non-zero $v \in W$. Write
\[
v = \sum_{k=0}^{p} a_k v_k \quad \text{with } a_k \in \k,
\]
and assume $a_p \neq 0$. As $U$ is a subcomodule, there exists $r \in \mathbb{Z}$ such that
\[
\delta(v) = g^r \otimes v = \sum_{l=0}^{p} a_l g^r \otimes v_l.
\]	
On the other hand, by the definition of the coaction,
\[
\delta(v) = \sum_{l=0}^{p} \left( \sum_{k=l}^{p} a_k c_j^i(k,l) \right) \otimes v_l.
\]
Comparing the coefficients of $v_0$, we obtain
\[
a_0 g^r = \sum_{k=0}^{p} a_k c_j^i(k,0) = \sum_{k=0}^{p} a_k \lambda_j^i(k,0) x^k g^{i - kt}.
\]
Since $a_p \neq 0$ and $\lambda_j^i(p,0) \neq 0$ (because $i \in \mathcal{J}$), the right-hand side is a polynomial in $x$ of degree $p$. For the equality to hold, we must have $p = 0$. Consequently, $v = a_0 v_0 \in W$, and thus $v_0 \in W$. It follows that $W = V$, which completes the proof.
\end{proof}

\begin{corollary}\label{cor:hasinfinitemodule}
All simple Yetter-Drinfeld modules over $H$ are finite-dimensional if and only if $t$ and $n$ are coprime.
\end{corollary}
\begin{proof}
This follows immediately from Corollary \ref{cor:coversectijequal0} and Lemma~\ref{lem:yanzhengydmoinfinit}
\end{proof}

\begin{proposition}\label{prop:infiniteisotoV}
Every infinite-dimensional simple Yetter-Drinfeld module $V$ over $H$ is isomorphic to $V(i,j)$ for some  $i \in \mathcal{J}$ and $j \in \mathbb{Z}$.
\end{proposition}

\begin{proof}
Let $v$ be a standard element of $V$; by Corollary \ref{cor:coversectijequal0}, $v \in V_{[j]}^{[i]}$ for some $i \in \mathcal{J}$ and $j \in \mathbb{Z}$. Lemma~\ref{thm:standbasisinfinite} implies that $\{x^k \cdot v \mid  k\geq 0   \}$ is a basis of $V$.	

Let $\left\{v_k \mid k \geq 0\right\}$ be the basis of $V(i, j)$ defined in~\eqref{basisofVinfinite}. Define a linear map $f: V(i, j) \to V$ by $f(v_k) = x^k \cdot v$ for all $ k\geq 0$. It is straightforward to verify that $f$ is an isomorphism of Yetter-Drinfeld modules, so $V \cong V(i, j)$.
\end{proof}

\begin{proposition}\label{prop:isoconditioninfinite}
For any $ i,r \in \mathcal{I} $ and $ j,s \in \Z $, $V(i,j) \cong  V(r,s) $ if and only if   $ i \equiv r  \pmod n  $ and  $ j \equiv s  \pmod n  $.
\end{proposition}

\begin{proof}
	Let $V = V(i,j)$ and $W = V(r,s)$ with bases $\{v_k \mid k \geq 0\}$ and $\{w_k \mid k \geq 0\}$ defined in \eqref{basisofVinfinite}, respectively.
	
	If $i \equiv r \pmod{n}$ and $j \equiv s \pmod{n}$, then clearly $V \cong W$.
	
	Conversely, suppose $V \cong W$ and let $f: V \to W$ be an isomorphism of Yetter-Drinfeld modules. Since $v_0 \in V_{[j]}^{[i]}$, its image $f(v_0) \in W_{[j]}^{[i]}$. Hence, $\operatorname{span}\{f(v_0)\}$ is a simple subcomodule of $W$. By the proof of Lemma~\ref{lem:yanzhengydmoinfinit}, the only simple subcomodule of $W$ is $\operatorname{span}\{w_0\}$. Therefore, $f(v_0) = \lambda w_0$ for some nonzero $\lambda \in \Bbbk$, which implies $f(v_0) \in W_{[r]}^{[s]}$. Consequently, $i \equiv r \pmod{n}$ and $j \equiv s \pmod{n}$, completing the proof.
\end{proof}

We conclude this section with the following theorem. 

\begin{theorem}\label{thm:YDmodule}
Let $ H=H(n,t,\xi) $. Then 
\begin{itemize}
	\item [(1)] All finite-dimensional simple Yetter-Drinfeld modules over $H$ are $ V(ti,j,\lambda) $, where $ i,j \in \Z $ and $ \lambda \in \k $. 
	\item [(2)] $ H $ has infinite-dimensional simple Yetter-Drinfeld modules if and only if $ t $ and $ n $ are not coprime. And in this case, all infinite-dimensional simple Yetter-Drinfeld modules over $H$ are $ V(i,j) $, where  $i \in \mathcal{J}$ and $j \in \mathbb{Z}$. 
	\item [(3)] The isomorphism relations between these  modules are given in Propositions \ref{prop:isocondition} and \ref{prop:isoconditioninfinite}.
\end{itemize}
\end{theorem}

\begin{proof}
This follows immediately from Lemmas~\ref{lem:yanzhengYDmo} and~\ref{lem:yanzhengydmoinfinit}, Propositions~\ref{prop:finiteisotoV} and~\ref{prop:infiniteisotoV}, and Corollary~\ref{cor:hasinfinitemodule}.
\end{proof}

\section{Finite-dimensional Nichols algebras over the simple modules $ V(ti,j,\lambda) $ }\label{section5}

In this section, we determine the pairs $(N, i, j, \lambda)$ such that $\B(V(ti,j,\lambda))$ is finite-dimensional, where $i, j \in \mathbb{Z}$, $\lambda \in \Bbbk$, and $N \geq 1$. We proceed by discussing three cases, depending on the values of $N$ and $\lambda$.

\subsection{The case $ N=1 $}

First, consider the case $N = 1$ (equivalently, $t = 0$), which implies $ti = 0$ for all $i \in \mathbb{Z}$.

For any $i, j \in \mathbb{Z}$ and $\lambda \in \Bbbk$, let $\{v_0, v_1, \dots, v_m\}$ be the basis of $V(0,j,\lambda)$ defined in \eqref{basisofV}. By Proposition~\ref{lem:Atijblocks}, we have $\overline{\mathcal{A}_j^{0}} = \mathcal{A}_j^{0}(0)$ and
\begin{equation*}
	\A_j^0(m)=
	\left(\begin{array}{cccc}
		1 &   \\[5pt]
		  & 1 &   \\[5pt]
		  &   & \ddots  \\[5pt]
		  &   &   & 1
	\end{array}
	\right).
\end{equation*}
For any $0 \leq k \leq m$, we have $\delta(v_k) = 1 \otimes v_k$, meaning $V(0,j,\lambda)$ is semisimple as a comodule. Moreover, for any $0 \leq k, l \leq m$, the braiding satisfies $c(v_k \otimes v_l) = v_l \otimes v_k$. Hence, by Remark~\ref{rmk:nicholalgebra}(2), the Nichols algebra $\B(V(0,j,\lambda))$ must be infinite-dimensional.

\begin{lemma}\label{lem:finiten-nichols-case1}
If $t = 0$, then the Nichols algebra $\B(V(ti,j,\lambda))$ is infinite-dimensional for all $i, j \in \mathbb{Z}$ and $\lambda \in \Bbbk$.
\end{lemma}

\subsection{The case $ N\geq 2 $ and $\lambda=0 $}

Let $ N\geq 2 $ (equivalently, $ 1\leq t\leq n-1 $) and $ \lambda=0 $. 
For any   $ i,j \in \Z $,  we will prove that  $ V(ti,j,0) $ is isomorphic to $ F(\mathcal{V}_{i,j}) $ as a braided vector space,  where $ F(\mathcal{V}_{i,j}) $  is a simple Yetter-Drinfeld module over the (finite-dimensional) Taft algebra.  Therefore, to determine whether $ \B(V(ti,j,0)) $ is finite-dimensional, it suffices to determine whether $ \B(F(\mathcal{V}_{i,j})) $ is finite-dimensional.

Let us begin by introducing $F(\mathcal{V}_{i,j})$. Recall that $w = \xi^t$ is a primitive $N$-th root of unity. The authors in~\cite{MBG21} describe the explicit structure of the simple Yetter-Drinfeld modules over the generalized Taft algebra $A(Nm,N,0,w^{-1})$ (see Remark \ref{rmk:cannotrealize}). Below, we summarize the case for $A(N,N,0,w^{-1})$.

For any $i, j \in \mathbb{Z}$, the module $F(\mathcal{V}_{i,j})$ is a simple Yetter-Drinfeld module over $A(N,N,0,w^{-1})$, defined as follows:

\begin{itemize}
	\item $\left\{v_0,v_1,\cdots,v_r \right\}  $ is a $ \k $-basis of  $ F(\mathcal{V}_{i,j}) $, where 
	\begin{equation*}
		r=\phi (i+j+1)-1;
	\end{equation*}
	\item The action of $g$ is defined for $0 \leq k \leq r$ by
	\begin{equation*}
		g\cdot v_k=w^{k-j} v_k;
	\end{equation*}
	\item The action of $x$ is defined for $0 \leq k \leq r$ by
	\begin{equation}\label{basisofFV}
		x\cdot v_k=
		\begin{cases}
			-w^{k-j}v_{k+1}, & 0\leq k<r, \\[5pt]
			0  ,     & k=r .
		\end{cases} 
	\end{equation}
	\item The coaction is defined for $0 \leq k \leq r$ by 
	\begin{equation*}
		\delta(v_k) =\sum\limits_{l=0}^{k} \beta_{k-l,k}^{i,j} x^{k-l}g^{-i+l} \otimes v_l,
	\end{equation*}
     where 
    \begin{equation*}\label{formulaofbeta}
     		\beta_{k-l,k}^{i,j}=
     	\begin{cases}
     		\binom{k}{l}_{w}\prod\limits_{s=l}^{k-1}(w^{-i}-w^{j-s}), & 0\leq l<k, \\[5pt]
     		1   ,    & l=k. 
     	\end{cases}  
    \end{equation*}
\end{itemize}

We cannot assert that $V(ti,j,0)$ and $F(\mathcal{V}_{i,j})$ are identical as Yetter-Drinfeld modules because they are modules over different Hopf algebras. Furthermore, in our definition of the infinite-dimensional  Taft algebra (Definition~\ref{def:infinite-Taft}), $x$ is required to be a $(g^t,1)$-primitive element, whereas here $x$ is defined as a $(1,g)$-primitive element. However, we will prove that the braiding induced by $V(ti,j,0)$ and the braiding induced by $F(\mathcal{V}_{i,j})$ are identical.

\begin{lemma}\label{lem:isoasbraidedspace}
For any $ i,j\in \Z $, $ V(ti,j,0) $ is isomorphic to $ F(\mathcal{V}_{i,j}) $ as braided vector spaces.
\end{lemma}

\begin{proof}
Let $\{v_0, v_1, \dots, v_m\}$ be the basis of $V(ti,j,0)$ defined in \eqref{basisofV}, where $m = N - \phi(-i-j)$. Similarly, let $\{u_0, u_1, \dots, u_r\}$ be the basis of $F(\mathcal{V}_{i,j})$ defined in \eqref{basisofFV}, where $r = \phi(i+j+1) - 1$. Denote the braidings on $V(ti,j,0)$ and $F(\mathcal{V}_{i,j})$ by $c_1$ and $c_2$, respectively.

From the definition of $\phi$, we observe that $\phi(i+j+1) = N - \phi(-i-j) + 1 = m + 1$, which implies $m = r$. For $k > m$, we set $v_k = u_k = 0$. For any $0 \leq p, k \leq m$, we have:
\begin{eqnarray*}
	c_1 (v_p \otimes v_k) 
	&=& (v_p)_{-1} \cdot v_k \otimes (v_p)_0  \\
	&=& \sum\limits_{l=0}^{p}  \lambda_j^{ti}(p,l)x^{p-l}g^{t(i-p)} \cdot v_k \otimes v_l   \\
	&=& w^{(i-p)(j-k)}  \sum\limits_{l=0}^{p}  \lambda_j^{ti}(p,l) v_{p+k-l} \otimes v_l.
\end{eqnarray*}

Similarly,
\begin{eqnarray*}
	& &c_2 (u_p \otimes u_k) \\
	&=& (u_p)_{-1} \cdot u_k \otimes (u_p)_0  \\
	&=& \sum\limits_{l=0}^{p}  \beta_{p-l,p}^{i,j} x^{p-l}g^{-i+l} \cdot u_k \otimes u_l \\
	&=& \sum\limits_{l=0}^{p} w^{(k-j)(-i+l)}  \beta_{p-l,p}^{i,j} (-1)^{p-l} \left(\prod\limits_{s=k}^{k+p-l-1} w^{s-j} \right)   u_{p+k-l} \otimes u_l  \\
	&=& \sum\limits_{l=0}^{p} w^{(k-j)(-i+l)}  \beta_{p-l,p}^{i,j} (-1)^{p-l}   w^{(k-j)(p-l)}\left(\prod\limits_{s=0}^{p-l-1} w^s \right)   u_{p+k-l} \otimes u_l  \\
	&=& w^{(i-p)(j-k)} \sum\limits_{l=0}^{p}   \beta_{p-l,p}^{i,j} (-1)^{p-l}  w^{\frac{(p-l)(p-l-1)}{2}}  u_{p+k-l} \otimes u_l  \\
    &=& w^{(i-p)(j-k)} \sum\limits_{l=0}^{p} \lambda_j^{ti}(p,l)   u_{p+k-l} \otimes u_l.
\end{eqnarray*}

The last equality follows from the identity:
\begin{eqnarray*}
	& & \beta_{p-l,p}^{i,j} (-1)^{p-l}  w^{\frac{(p-l)(p-l-1)}{2}} \\
	&=& w^{\frac{(p-l)(p-l-1)}{2}} (-1)^{p-l} \binom{p}{l}_{w} \left(\prod\limits_{r=l}^{p-1}(w^{-i}-w^{j-r})  \right)  \\
	&=& w^{\frac{(p-l)(p-l-1)}{2}} \binom{p}{l}_{w} \left(\prod\limits_{r=l}^{p-1}(w^{j-r}-w^{-i})  \right) \\
	&=& \binom{p}{l}_{w} w^{\frac{(p-l)(p-l-1)}{2}}  \left(\prod\limits_{r=l+1}^{p}  w^{-r+1} \right)  \left(\prod\limits_{r=l+1}^{p}  (w^j-w^{r-1-i})  \right)\\
	&=& \binom{p}{l}_{w} w^{\frac{(p-l)(p-l-1)}{2}}  w^{-\frac{(p-l)(p+l-1)}{2}}    \left(\prod\limits_{r=l+1}^{p} R_j^{ti}(r,0) \right)  \\
	&=& \binom{p}{l}_{w} w^{-(p-l)l}    \left(\prod\limits_{r=l+1}^{p} R_j^{ti}(r,0) \right)   \\
	&=& \lambda_j^{ti}(p,l).
\end{eqnarray*}

Define a linear map $f: V(ti,j,0) \to F(\mathcal{V}_{i,j})$ by $f(v_k) = u_k$ for all $0 \leq k \leq m$. The above calculations show that $f$ preserves the braiding, i.e., $(f \otimes f) \circ c_1 = c_2 \circ (f \otimes f)$. Since $f$ is bijective, it is an isomorphism of braided vector spaces.
\end{proof}

Now, it suffices to determine whether  $ \B( F(\mathcal{V}_{i,j}) ) $ is finite-dimensional. Using results of Andruskiewitsch and Angiono \cite{AA20}, the authors of \cite{MBG21} reduce this problem to determining whether a Nichols algebra of diagonal type is finite-dimensional.

Write $ C_N $  for the cyclic
group of order $ N $ with generator $ g $. For any $ i,j \in \Z $, let $ W_{i,j}  $ be the Yetter-Drinfeld module over $ C_N $, defined as follows:

\begin{itemize}
	\item $\left\{x,y \right\}  $ is a $ \k $-basis of  $ W_{i,j} $;
	\item The action of $g$ is defined by
	\begin{equation*}
		g\cdot x=wx \ ,\  g \cdot y=w^j y ;
	\end{equation*}

	\item The coaction is defined by
	\begin{equation*}
		\delta(x)=g\otimes x \ ,\  \delta(y)=g^i \otimes y  .
	\end{equation*}	
\end{itemize}

By Lemma \ref{lem:isoasbraidedspace} and \cite[Theorem 5.3]{MBG21}, we obtain the following result.

\begin{lemma}\label{lem:VijandWij}
For any $ i,j\in \Z $, it holds that $  \B(V(ti,j,0)) $ is finite-dimensional if and only if $\B(W_{-i,-j}) $ is finite-dimensional.	
\end{lemma}

It remains to classify all pairs $(N, i, j)$ such that $\B(W_{-i,-j})$ is finite-dimensional.
We observe that $W_{-i,-j}$ is a braided vector space of diagonal type associated with the basis $\left\{x,y \right\}  $. Hence, the finite-dimensionality of its Nichols algebra can be decided by appealing to Heckenberger's classification \cite{H08,H09}, which establishes a correspondence between finite-dimensional Nichols algebras of diagonal type and certain generalized Dynkin diagrams. We now proceed to systematically apply this theory.

The braiding $ c $ associated with $ W_{-i,-j} $ is given by:
\begin{eqnarray*}
	& & c(x\otimes x)=w x\otimes x , \  c(x\otimes y)= w^{-j} y \otimes x,  \\
	& & c(y\otimes x)=w^{-i} x \otimes y , \  c(y\otimes y)=w^{ij} y\otimes y.
\end{eqnarray*}

Hence, the corresponding braiding matrix and generalized Dynkin diagram $ D_{i,j} $ are 
\begin{eqnarray*}
	\boldsymbol{q}
	&=& \left(\begin{array}{cccc}
		        w &   w^{-j}      \\
		        w^{-i}  & w^{ij}
	          \end{array}    
         \right),   
\end{eqnarray*}
\begin{eqnarray}\label{Dynkindiagram}
D_{i,j}: \quad 
\begin{tikzpicture}[
	scale=1.0,
	every node/.style={
		circle, 
		draw=black,        
		fill=white,        
		inner sep=2pt,
	},
	label distance=0.4mm     
	]
	\node[label=above:$w$] (A) at (0,0) {};
	\node[label=above:$w^{ij}$]  (B) at (2,0) {};	
\end{tikzpicture}
\qquad \quad  \text{or} \qquad \quad
D_{i,j}: \quad 
\begin{tikzpicture}[
	scale=1.0,
	every node/.style={
		circle, 
		draw=black,        
		fill=white,        
		inner sep=2pt,
	},
	label distance=0.4mm     
	]
	\node[label=above:$w$] (A) at (0,0) {};
	\node[label=above:$w^{ij}$]  (B) at (2,0) {};
	
	\draw (A) --node[fill=none,draw=none,above,yshift=-5pt]{$ w^{-i-j} $} (B) ;
	
\end{tikzpicture}.
\end{eqnarray}


The generalized Dynkin diagrams corresponding to arithmetic root systems of
rank 2 are listed in \cite[Table 1]{H08}. These diagrams depend on fixed parameters $ q,\zeta \in \k^* $.  Following the notation in \cite{MBG21}, we refer to these as Heckenberger diagrams. Each diagram is indexed as $H_{k,l}$, where $k$ denotes the row number and $l$ denotes the position within that row from left to right.
For example, $ H_{4,1} $ denotes the diagram 
$ \begin{tikzpicture}[
	scale=0.7,
	every node/.style={
		circle, 
		draw=black,        
		fill=white,        
		inner sep=1pt,
	},
	label distance=0.1mm     
	]
	\node[label=above:$q$] (A) at (0,0) {};
	\node[label=above:$q^{2}$]  (B) at (2,0) {};
	
	\draw (A) --node[fill=none,draw=none,above,yshift=-5pt]{$ q^{-2} $} (B) ;
	
\end{tikzpicture} $
, with $ q\in \k^* \textbackslash \{-1,1\} $. The set of indices of Heckenberger
diagrams is the following:
\begin{eqnarray*}
	\mathcal{I} \ := & \{& (1,1),(2,1),(3,1),(3,2),(4,1),(5,1),(5,2),(6,1),(6,2),(7,1),(7,2), \\   
	&&(8,1), (8,2),(8,3),(8,4),(8,5),(9,1),(9,2),(9,3),(10,1), (10,2),(10,3),     \\
	&& (11,1),(12,1),(12,2),(12,3),(13,1),(13,2),(13,3),(13,4),(14,1),(14,2),   \\
	&& (15,1),(15,2),(15,3),(15,4),(16,1),(16,2),(16,3),(16,4), (17,1),(17,2) \  \} .
\end{eqnarray*}

By invoking \cite[Corollary 6]{H08}, we can determine whether $ \B(W_{-i,-j}) $ is finite-dimensional through $ D_{i,j} $, as formalized in the following lemma. 
\begin{lemma}\label{lem:WijandDij}
For any $ i,j\in \Z $, it holds that $ \B(W_{-i,-j}) $ is finite-dimensional if and only if $
D_{i,j}=H_{k,l}  $ for some $ (k,l)\in \mathcal{I} $ and $ w^{ij}\neq 1 $.	
\end{lemma}
In fact, $ w^{ij}\neq 1 $ automatically when $ (k,l)\neq (1,1) $. In these cases, $ \B(W_{-i,-j}) $ is finite-dimensional if and only if $D_{i,j}=H_{k,l}  $.  Therefore, our primary task reduces to classifying all pairs $ (N,i,j) $ such that $ D_{i,j} $ is a Heckenberg diagram. We first present all possible pairs $ (k,l) $. Consider the following subset of indices of $ \mathcal{I} $:
\begin{eqnarray*}
	\mathcal{L} \ := & \{& (1,1),(2,1),(3,1),(4,1),(5,1),(5,2),(6,1),(6,2),(9,3),(10,1), (10,3),     \\
	&& (11,1),(12,1),(12,3),(13,2),(13,4),(14,2),   (15,1),(15,2),(16,1),(16,2)  \  \} .
\end{eqnarray*}




\begin{lemma}\label{lem:mustinL}
Let $ i,j \in \Z $ and assume $ D_{i,j} $ is a Heckenberger diagram $ H_{k,l} $. Then $ (k,l) \in \mathcal{L} $.
\end{lemma}

\begin{proof}
We show that the remaining cases are not possible. 	First consider the cases $ H_{17,1} $ and $ H_{17,2} $. Suppose that 
$ D_{i,j}=H_{17,1}:
\begin{tikzpicture}[
	scale=1.0,
	every node/.style={
		circle, 
		draw=black,        
		fill=white,        
		inner sep=1pt,
	},
	label distance=0.1mm     
	]
	\node[label=above:$-\zeta$] (A) at (0,0) {};
	\node[label=above:$-1$]  (B) at (2,0) {};
	
	\draw (A) --node[fill=none,draw=none,above,yshift=-5pt]{$ -\zeta^{-3} $} (B) ;
	
\end{tikzpicture} $, with $ \zeta \in R_7 $. Then $ w=-\zeta, w^{ij}=-1 $ and $ w^{-i-j}=-\zeta^3 $.
This implies that $ N=14, ij   \equiv 7 \pmod {14}  $ and $  -i-j   \equiv -3  \pmod {14} $. Consequently we have $ i(3-i) \equiv 7  \pmod {14} $. However, this equation admits no solution, leading to a contradiction. Hence $ D_{i,j} $ cannot be equal to $ H_{17,1} $.  The case $ H_{17,2} $ can be treated similarly.
	
Let us now turn to the cases in $ \mathcal{K}= \mathcal{I} \ \textbackslash  \left(\mathcal{L} \cup \left\{(17,1),(17,2)\right\} \right) $.
Note that $ D_{i,j} $ consists of two vertices and one edge, all labelled with roots of unity, and the root of unity at one vertex generates the labels of the other vertex and the edge. Therefore it suffices to prove that for all  $ (k,l)\in \mathcal{K}  $, $ H_{k,l} $ does not possess this property.  Let us prove this for the case $ H_{8,1} $; the other cases follow similarly.

$ H_{8,1} $ denotes the diagram 
$ \begin{tikzpicture}[
	scale=1.0,
	every node/.style={
		circle, 
		draw=black,        
		fill=white,        
		inner sep=1pt,
	},
	label distance=0.1mm     
	]
	\node[label=above:$-\zeta^{-2}$] (A) at (0,0) {};
	\node[label=above:$-\zeta^{2}$]  (B) at (2,0) {};
	
	\draw (A) --node[fill=none,draw=none,above,yshift=-4pt]{$ -\zeta^{3} $} (B) ;
	
\end{tikzpicture} $, with $ \zeta\in R_{12} $. Note that $-\zeta^{-2}  $ and
$-\zeta^{2}  $ are primitive 3rd roots of unity, while 
$-\zeta^{3}  $  is a primitive 4th root of unity; consequently, neither  $-\zeta^{-2}  $ nor
$-\zeta^{2}  $  can generate $-\zeta^{3}  $. Hence $ D_{i,j} $ cannot be equal to $ H_{8,1} $.

\end{proof}

In the following lemma, we state the converse of Lemma \ref{lem:mustinL}. In particular,
we present the list of all pairs $ (N,i,j) $ such that $ \B (V(ti,j,0)) $ is finite-dimensional, and the
conditions that the integers $ N, i $ and $ j $ must satisfy.

\begin{lemma}\label{lem:finiten-nichols-case2}
For all pairs $(k,l) \in \mathcal{L}$, there exists a triple $(N, i, j)$ with $N \geq 2$ and $i,j \in \mathbb{Z}$ such that $D_{i,j} = H_{k,l}$. Moreover, for any such triple $(N, i, j)$, it holds that $ \dim_{\k} \B (V(ti,j,0)) < \infty $ if and only if  $ ij\not \equiv 0 \pmod N $.  The complete classification of triples $(N, i, j)$ yielding finite-dimensional Nichols algebras appears in the following Table 1-6.
\end{lemma}

\begin{proof}
We first provide an explanation of the table below.
Note that if  $\dim_{\k} V(ti,j,0)=k $, then $ 1 \leq k \leq N $ and $ j\equiv k-1-i \pmod N $. 
In each table below, the dimension of $ V(ti,j,0) $ is fixed, hence $ j $ is entirely determined by $ N $ and $ i $. When $ N $ satisfies the condition in the first column and $ i $ satisfies the condition in the second column, $ \B (V(ti,j,0)) $ is finite-dimensional, and 
$ D_{i,j}=H_{k,l} $ for any $ (k,l) $ belonging to the third column of the table.

From Lemma \ref{lem:VijandWij} and Lemma \ref{lem:WijandDij}, it follows that $ \B (V(ti,j,0)) $ is finite-dimensional if and only if $
D_{i,j}=H_{k,l}  $ for some $ (k,l)\in \mathcal{I} $ and $ {ij} \not \equiv 0 \pmod N  $.	
We now classify all triples $ (N,i,j) $ such that  $ D_{i,j}=H_{k,l} $, where $ (k,l) \in \mathcal{L} $.
We address only the first 4 cases and the case $ (12,1) $, the remaining ones follow mutatis mutandis.
\\
\textbf{Case  (1,1):} $ H_{1,1}=
\begin{tikzpicture}[
	scale=0.7,
	every node/.style={
		circle, 
		draw=black,        
		fill=white,        
		inner sep=1pt,
	},
	label distance=0.1mm     
	]
	\node[label=above:$q$] (A) at (0,0) {};
	\node[label=above:$r$]  (B) at (2,0) {};	
\end{tikzpicture} $, with $ q,r \in k^* $. Let $ (i, j) $ be such that $ D_{i,j}=H_{1,1} $. Then $ w=q, w^{ij}=r $ and $ w^{-i-j}=1 $.
This implies that  $  -i-j   \equiv 0  \pmod N $. In this case, $ \dim_\k (V(ti,j,0))=1  $. 
\\
\textbf{Case  (2,1):} $ H_{2,1}=
\begin{tikzpicture}[
	scale=0.7,
	every node/.style={
		circle, 
		draw=black,        
		fill=white,        
		inner sep=1pt,
	},
	label distance=0.1mm     
	]
	\node[label=above:$q$] (A) at (0,0) {};
	\node[label=above:$q$]  (B) at (2,0) {};
	\draw (A) --node[fill=none,draw=none,above,yshift=-5pt]{$ q^{-1} $} (B) ;	
\end{tikzpicture} $, with $ q \in k^* \textbackslash \left\{1\right\} $. Let $ (i, j) $ be such that $ D_{i,j}=H_{2,1} $. Then $ w=q, w^{ij}=q $ and $ w^{-i-j}=q^{-1} $.
This implies that $ ij   \equiv 1  \pmod N $ and $  -i-j   \equiv -1  \pmod N $. In this case, $ \dim_\k (V(ti,j,0))=2  $. 
\\
\textbf{Case  (3,1):} $ H_{3,1}=
\begin{tikzpicture}[
	scale=0.7,
	every node/.style={
		circle, 
		draw=black,        
		fill=white,        
		inner sep=1pt,
	},
	label distance=0.1mm     
	]
	\node[label=above:$q$] (A) at (0,0) {};
	\node[label=above:$-1$]  (B) at (2,0) {};
	\draw (A) --node[fill=none,draw=none,above,yshift=-5pt]{$ q^{-1} $} (B) ;	
\end{tikzpicture} $, with $ q \in k^* \textbackslash \left\{-1,1\right\} $. Let $ (i, j) $ be such that $ D_{i,j}=H_{3,1} $. Then $ w=q, w^{ij}=-1 $ and $ w^{-i-j}=q^{-1} $.
This implies that $ N=2k $ with $ k\geq 2 $, $ ij   \equiv k  \pmod N $ and $  -i-j   \equiv -1  \pmod N $. In this case, $ \dim_\k (V(ti,j,0))=2  $. 
\\
\textbf{Case  (4,1):} $ H_{4,1}=
\begin{tikzpicture}[
	scale=0.7,
	every node/.style={
		circle, 
		draw=black,        
		fill=white,        
		inner sep=1pt,
	},
	label distance=0.1mm     
	]
	\node[label=above:$q$] (A) at (0,0) {};
	\node[label=above:$q^2$]  (B) at (2,0) {};
	\draw (A) --node[fill=none,draw=none,above,yshift=-5pt]{$ q^{-2} $} (B) ;	
\end{tikzpicture} $, with $ q \in k^* \textbackslash \left\{-1,1\right\} $. Let $ (i, j) $ be such that $ D_{i,j}=H_{4,1} $. The first case is $ w=q, w^{ij}=q^2 $ and $ w^{-i-j}=q^{-2} $.
This implies that $ N\geq 3 $, $ ij   \equiv 2  \pmod N $ and $  -i-j   \equiv -2  \pmod N $. In this case, $ \dim_\k (V(ti,j,0))=3  $.  
 
Another case is $ w=q^2, w^{ij}=q $ and $ w^{-i-j}=q^{-2} $.
This implies that $ N=2k+1 $ with $ k\geq 1 $, $ ij   \equiv k+1  \pmod N $ and $  -i-j   \equiv -1  \pmod N $. In this case, $ \dim_\k (V(ti,j,0))=2  $.   
\\
\textbf{Case  (12,1):} $ H_{12,1}=
\begin{tikzpicture}[
	scale=0.7,
	every node/.style={
		circle, 
		draw=black,        
		fill=white,        
		inner sep=1pt,
	},
	label distance=0.1mm     
	]
	\node[label=above:$\zeta^2$] (A) at (0,0) {};
	\node[label=above:$\zeta^{-1}$]  (B) at (2,0) {};
	\draw (A) --node[fill=none,draw=none,above,yshift=-1pt]{$ \zeta $} (B) ;	
\end{tikzpicture} $, with $ \zeta \in R_8 $. Let $ (i, j) $ be such that $ D_{i,j}=H_{12,1} $. Then $ w=\zeta^{-1}, w^{ij}=\zeta^2 $ and $ w^{-i-j}=\zeta $.
This implies that $ N=8 $, $ ij   \equiv -2  \pmod 8 $ and $  -i-j   \equiv -1  \pmod 8 $. Consequently we have $ i\equiv 2,7  \pmod 8 $. In this case, $ \dim_\k (V(ti,j,0))=2  $.

\begin{center}
	
\title{TABLE 1. $ \dim_\k (V(ti,j,0))=1 $}

\scalebox{1.0}{
\begin{tabular}{|>{\centering\arraybackslash}m{3.7cm}|>{\centering\arraybackslash}m{5.8cm}|>{\centering\arraybackslash}m{2.8cm}|}
\hline   
Conditions on $ N $  & Conditions on  $ i $  & $(k,l):D_{i,j}=H_{k,l}$ \\
\hline 
$N\geq 2 $ & $ i^2 \not\equiv 0 \pmod N  $ & $ (1,1) $  \\
\hline
\end{tabular}}

\title{TABLE 2. $ \dim_\k (V(ti,j,0))=2 $}
\scalebox{1.0}{
\begin{tabular}{|>{\centering\arraybackslash}m{3.7cm}|>{\centering\arraybackslash}m{5.8cm}|>{\centering\arraybackslash}m{2.8cm}|}
\hline   
Conditions on $ N $  & Conditions on  $ i $  & $(k,l):D_{i,j}=H_{k,l}$ \\
\hline 
$N\geq 2 $ & $  i(1-i) \equiv 1  \pmod N $ & $ (2,1) $ \\
\hline
$N=2k $ with $ k\geq 2 $    & $  i(1-i) \equiv k  \pmod N $ & $ (3,1) $ \\
\hline
$N=2k+1 $ with $ k\geq 1 $   & $  i(1-i) \equiv k+1  \pmod N $ & $ (4,1) $ \\
\hline
$N=3k $ with $ k\geq 2 $   & $  i(1-i) \equiv k , 2k  \pmod N $ & $ (6,1),(6,2) $ \\
\hline
$N=3k+1 $  with $ k\geq 1 $   &  $  i(1-i) \equiv 2k+1   \pmod N $   & $ (11,1) $ \\
\hline
$N=3k+2 $  with $ k\geq 1 $   &  $  i(1-i) \equiv k+1   \pmod N $   & $ (11,1) $ \\
\hline
$N=8 $     &  $  i \equiv 2 , 7   \pmod 8 $   & $ (12,1) $ \\
\hline
$N=24 $     &  $  i \equiv 3,6,19, 22   \pmod {24} $   & $ (13,2) $ \\
\hline
$N=30 $     &  $  i \equiv 4, 7, 9, 12, 19, 22, 24, 27   \pmod {30} $   & $ (16,2) $ \\
\hline
\end{tabular}}

\title{TABLE 3. $ \dim_\k (V(ti,j,0))=3 $}
\scalebox{1.0}{
\begin{tabular}{|>{\centering\arraybackslash}m{3.7cm}|>{\centering\arraybackslash}m{5.8cm}|>{\centering\arraybackslash}m{2.8cm}|}
\hline   
Conditions on $ N $  & Conditions on  $ i $  & $(k,l):D_{i,j}=H_{k,l}$ \\
\hline 
$N\geq 3 $ & $  i(2-i) \equiv 2  \pmod N $ & $ (4,1) $ \\
\hline
$N=2k $ with $ k\geq 3 $  & $  i(2-i) \equiv k  \pmod N $ & $ (5,1),(5,2) $ \\
\hline
$N=18 $     &  $  i \equiv 6,14   \pmod {18} $   & $ (10,1) $ \\
\hline
\end{tabular}}

\title{TABLE 4. $ \dim_\k (V(ti,j,0))=4 $}
\scalebox{1.0}{
\begin{tabular}{|>{\centering\arraybackslash}m{3.7cm}|>{\centering\arraybackslash}m{5.8cm}|>{\centering\arraybackslash}m{2.8cm}|}
\hline   
Conditions on $ N $  & Conditions on  $ i $  & $(k,l):D_{i,j}=H_{k,l}$ \\
\hline 
$N=12 $ & $  i \equiv 6,9  \pmod {12}$ & $ (9,3) $ \\
\hline
$N\geq 4 $   & $  i(3-i) \equiv 3  \pmod N $ & $ (11,1) $ \\
\hline
$N=8 $    & $  i \equiv 4,7  \pmod 8 $  & $ (12,3) $ \\
\hline
$N=20 $     &  $  i \equiv 5,10,13,18   \pmod {20} $   & $ (15,1),(15,2) $ \\
\hline
$N=30 $     &  $  i \equiv 5,8,10,13,20,23,25,28  \pmod {30} $  & $ (16,1) $ \\
\hline
\end{tabular}}

\title{TABLE 5. $ \dim_\k (V(ti,j,0))=5 $}
\scalebox{1.0}{
\begin{tabular}{|>{\centering\arraybackslash}m{3.7cm}|>{\centering\arraybackslash}m{5.8cm}|>{\centering\arraybackslash}m{2.8cm}|}
\hline   
Conditions on $ N $  & Conditions on  $ i $  & $(k,l):D_{i,j}=H_{k,l}$ \\
\hline 
$N=18 $ & $  i \equiv 9,13  \pmod {18}$ & $ (10,3) $ \\
\hline
$N=10 $   & $  i \equiv 5,9  \pmod {10} $ & $ (14,2) $ \\
\hline
	
\end{tabular}}

\title{TABLE 6. $ \dim_\k (V(ti,j,0))=6 $}
\scalebox{1.0}{
\begin{tabular}{|>{\centering\arraybackslash}m{3.7cm}|>{\centering\arraybackslash}m{5.8cm}|>{\centering\arraybackslash}m{2.8cm}|}
\hline   
Conditions on $ N $  & Conditions on  $ i $  & $(k,l):D_{i,j}=H_{k,l}$ \\
\hline
$N=24 $ & $  i \equiv 9,12,17,20  \pmod {24}$ & $ (13,4) $ \\
\hline
\end{tabular}}
\end{center}
\end{proof}

\subsection{The case $ N\geq 2 $ and $\lambda\neq 0 $}

We now turn to the final case: $ N\geq 2 $ and $\lambda\neq 0 $. Under these conditions, we show that all Nichols algebras of the form  $ \B(V(ti,j,\lambda)  $ are infinite-dimensional.

Fix $ i,j\in \Z $ and $ 0\neq \lambda \in \k $, and let $ V = V(ti,j,\lambda) $. By Proposition \ref{prop:isocondition}, we may assume without loss of generality that $ 0\leq i \leq n-1 $. Let $\{v_0, v_1, \dots, v_{n-1}\}$ be the basis of $V$ defined in~\eqref{basisofV}. For any $ 0\leq k \leq n-1 $, define the one-dimensional subspace
$ V_k=\span \left\{v_k\right\} $ and the projection
\[
p_k : V \to V_k, \quad \sum_{l=0}^{n-1} a_l v_l \mapsto a_k v_k.
\]
For any  $ u_1, \dots, u_m \in V $ , we write  $ u_1 \dotsm u_m $  to denote the tensor product  $ u_1 \otimes \cdots \otimes u_m $; that is, we suppress the tensor symbol for brevity.
We now prove that for all  $ m \geq 1 $,
\begin{equation}
	\Delta_{1^m}^{T(V)}(v_i^m) \neq 0,
\end{equation}
which implies that  $ \B(V) $  is infinite-dimensional.

For all $ m\geq 1 $, define $ \psi_m =p_i^{\otimes m}\circ  \Delta_{1^m} $. It suffices to show that 
\begin{equation}\label{psinot0}
	\psi_m(v_i^m) \neq 0.
\end{equation}
Note that $\psi_m(v_i^m) = a_m  v_i \otimes \cdots \otimes v_i$ for some scalar $a_m \in \Bbbk$.
Essentially, we aim to prove that the coefficient $a_m$ of the term $v_i \otimes \cdots \otimes v_i$ in the expansion of $\Delta_{1^m}^{T(V)}(v_i^m)$ is nonzero. 
To this end, we now perform a detailed computation of  $\Delta_{1^m}^{T(V)}(v_i^m)$.

For every $ m\geq 2 $, the following identity holds:
\begin{equation}\label{formuladelta}
	\Delta_{1^m}^{T(V)}(v_i^m)=(\pi_1 \otimes\Delta_{1^{m-1}}^{T(V)} )\circ \Delta (v_i^m).
\end{equation}
Since $\Delta : T(V) \to T(V) \underline{\otimes} T(V)$ is an algebra homomorphism,
\begin{equation*}
	\Delta (v_i^m) = (1\otimes v_i + v_i \otimes 1)^m = \sum x_1 x_2 \cdots x_m,
\end{equation*}
where the sum is taken over all sequences $(x_1, \dots, x_m)$ such that each $x_k$ is either $1 \otimes v_i$ or $v_i \otimes 1$. In particular, if exactly $l$ of the $x_k$ equal $v_i \otimes 1$ and the remaining $m-l$ equal $1 \otimes v_i$, then the product $x_1 x_2 \cdots x_m$ lies in $V^{\otimes l} \otimes V^{\otimes (m-l)}$.

By \eqref{formuladelta}, we need only compute the terms that lie in $V \otimes V^{\otimes (m-1)}$. That is, among all such products $x_1 x_2 \cdots x_m$, we restrict to those in which exactly one factor equals $v_i \otimes 1$. Therefore,
\begin{equation}\label{formuladelta11}
	\Delta_{1^m}^{T(V)}(v_i^m)=(\pi_1 \otimes\Delta_{1^{m-1}}^{T(V)} ) \left(  \sum_{k=1}^{m} (1\otimes v_i^{k-1})(v_i \otimes v_i^{m-k}) \right).
\end{equation}

We now compute the term $(1\otimes v_i^{k-1})(v_i \otimes v_i^{m-k})$. When $k = 1$, it reduces to $v_i \otimes v_i^{m-1}$. For $k > 1$, however, the computation becomes more involved. Specifically,
\begin{equation*}
	(1\otimes v_i^{k-1})(v_i \otimes v_i^{m-k})=(v_i^{k-1})_{-1} \cdot v_i \otimes (v_i^{k-1})_{0}v_i^{m-k}=\sum v_{j_1}\otimes v_{j_2}\cdots v_{j_m},
\end{equation*}
where $0 \leq j_l \leq n-1$ for all $1 \leq l \leq m$. As $k$ increases, so does the number of terms in this expansion. From the action and coaction structure on $V$, one can verify that for each term $v_{j_1} \otimes v_{j_2} \cdots v_{j_m}$, the following inequality holds:
\begin{equation*}
	j_1 + \cdots + j_m \leq m i.
\end{equation*}
In fact, we are able to prove the following generalization.
\begin{lemma}\label{lem:zhibiaobuzeng}
For any $m \geq 2$, suppose
\[
	\Delta_{1^m}^{T(V)}(v_{i_1}\otimes \cdots \otimes v_{i_m})=\sum v_{j_1}\otimes \cdots \otimes v_{j_m},
\]
where $0 \leq i_l,j_l \leq n-1$ for all $1 \leq l \leq m$. Then, for each term $v_{j_1} \otimes \cdots \otimes v_{j_m}$ in the sum, the following inequality holds:
\[
	j_1 + \cdots + j_m \leq i_1 + \cdots + i_m  .
\]
\end{lemma}

We now proceed to prove \eqref{psinot0}. The following derivation holds:
\begin{eqnarray*}
	& & \psi_m(v_i^m) \\
	&=&  (p_i^{\otimes m}\circ  \Delta_{1^m}) (v_i^m) \\
	&=& \left(p_i \otimes p_i^{\otimes^{m-1}}   \right) (\pi_1 \otimes\Delta_{1^{m-1}}^{T(V)} ) \left(  \sum_{k=1}^{m} (1\otimes v_i^{k-1})(v_i \otimes v_i^{m-k}) \right) \\
	&=& \sum_{k=1}^{m} \left(p_i  \pi_1 \otimes \psi_{m-1} \right) \left( (v_i^{k-1})_{-1} \cdot v_i \otimes (v_i^{k-1})_{0}v_i^{m-k}  \right).  
\end{eqnarray*}
For any $1 \leq k \leq m$, we claim that
\begin{equation*}
\left(p_i  \pi_1 \otimes \psi_{m-1} \right) \left( (v_i^{k-1})_{-1} \cdot v_i \otimes (v_i^{k-1})_{0}v_i^{m-k}  \right)=v_i\otimes  \psi_{m-1}(v_i^{m-1}) .
\end{equation*}
The case $k = 1$ is straightforward. For $k \geq 2$, we illustrate the argument for $k = 3$; the general case follows a similar but more intricate computation.  Note that $c_j^{ti}(i,i) = 1$, hence
\begin{eqnarray*}
	& & \left(p_i  \pi_1 \otimes \psi_{m-1} \right) \left( (v_i^{2})_{-1} \cdot v_i \otimes (v_i^{2})_{0}v_i^{m-3}  \right) \\
	&=&  \left(p_i  \pi_1 \otimes \psi_{m-1} \right) \left( \sum_{0\leq r,l \leq i} c_j^{ti}(i,r)c_j^{ti}(i,l) \cdot v_i \otimes v_rv_l	v_i^{m-3}
	\right) \\
	&=& v_i\otimes  \psi_{m-1}(v_i^{m-1}).
\end{eqnarray*}
The final equality holds because for $r < i$ or $l < i$, we have $\psi_{m-1}(v_r v_l v_i^{m-3}) = 0$ by Lemma \ref{lem:zhibiaobuzeng}. Therefore,
\begin{eqnarray*}
	& & \psi_m(v_i^m) \\
	&=&  \sum_{k=1}^{m} \left(p_i  \pi_1 \otimes \psi_{m-1} \right) \left( (v_i^{k-1})_{-1} \cdot v_i \otimes (v_i^{k-1})_{0}v_i^{m-k}  \right) \\
	&=& m \left(v_i\otimes  \psi_{m-1}(v_i^{m-1}) \right) \\
	&=& m (m-1) \left(v_i\otimes  v_i  \otimes \psi_{m-2}(v_i^{m-2}) \right) \\
	&=& m!( v_i \otimes \cdots \otimes v_i),
\end{eqnarray*}
which completes the proof. We thus obtain the following result.
\begin{lemma}\label{lem:finiten-nichols-case3}
If $1\leq t \leq n-1$, then the Nichols algebra $\B(V(ti,j,\lambda))$ is infinite-dimensional for all $i, j \in \mathbb{Z}$ and $0\neq\lambda \in \Bbbk$.
\end{lemma}

We conclude this section with the following theorem. 

\begin{theorem}\label{thm:nicholsalgebra}
Let $ H=H(n,t,\xi) $. Then 
\begin{itemize}
	\item [(1)] If $ t=0 $, then $ \B(V) $ is infinite-dimensional for every simple Yetter-Drinfeld module $ V $ over $ H $.
	\item [(2)] If $ 1\leq t \leq n-1 $, then the simple Yetter-Drinfeld modules $ V $ over $ H $ for which $ \B(V) $ is finite-dimensional are precisely those listed in Table 1-6 following Lemma \ref{lem:finiten-nichols-case2}.
\end{itemize}
\end{theorem}

\begin{proof}
This follows immediately from Lemmas~\ref{lem:finiten-nichols-case1}, \ref{lem:finiten-nichols-case2}, and \ref{lem:finiten-nichols-case3}.
\end{proof}



\bibliographystyle{plain}
\bibliography{Taftref.bib}

\end{document}